\documentclass{amsart}
\usepackage[utf8]{inputenc}
\usepackage{color}

\usepackage{listings}
\usepackage[margin=1.0in]{geometry}

\usepackage{amsmath}
\usepackage{amsthm}
\usepackage{amssymb}
\usepackage[matrix,arrow,curve]{xy}
\usepackage{enumitem} 
\usepackage{graphicx}
\usepackage{mdwlist}	
\usepackage{mathrsfs}
\usepackage{tikz}
\usepackage{hyperref}
\usetikzlibrary{patterns}
\usepackage{caption}
\usepackage{subcaption}
\usepackage{makecell}
\usepackage{graphicx}
\usepackage{import}
\usepackage{xifthen}
\usepackage{pdfpages}
\usepackage{transparent}
\newcommand\setItemnumber[1]{\setcounter{enum\romannumeral\@enumdepth}{\numexpr#1-1\relax}}

\newcommand{\p}{\mathbb{P}}

\DeclareMathOperator{\PGL}{PGL}

\DeclareMathOperator{\SL}{SL}

\DeclareMathOperator{\RR}{\mathbb{R}}
\DeclareMathOperator{\ZZ}{\mathbb{Z}}

\DeclareMathOperator{\CC}{\mathbb{C}}

\DeclareMathOperator{\Q}{\mathscr{Q}}

\DeclareMathOperator{\X}{\mathscr{X}}
\DeclareMathOperator{\Y}{\mathscr{Y}}

\DeclareMathOperator{\I}{\mathcal{I}}

\DeclareMathOperator{\Ind}{Ind}
\DeclareMathOperator{\Exc}{Exc}

\newcommand{\tC}{\widetilde{C}}

\newcommand{\tG}{\widetilde{\Gamma}}

\newcommand{\tsi}{\widetilde{\sigma}}

\newcommand{\xdashrightarrow}[2][]{\ext@arrow 0359\rightarrowfill@@{#1}{#2}}

\newtheorem{theorem}[equation]{Theorem}
\newtheorem{lemma}[equation]{Lemma}
\newtheorem{proposition}[equation]{Proposition}
\newtheorem{corollary}[equation]{Corollary}

\theoremstyle{definition}

\newtheorem{notation}[equation]{Notation}

\newtheorem{condition}{Condition}

\makeatletter
\@addtoreset{equation}{section}
\makeatother

\title{Regularizations of positive entropy pseudo-automorphisms}
\author{Alexandra Kuznetsova}

\address{
National research university Higher School of economics, Russia, Usacheva str. 6, 119048;
\'Ecole Polytechnique, France, CMLS, Route de Saclay, 91128 Palaiseau.}

 \email{sasha.kuznetsova.57@gmail.com}
\begin{document}
\begin{abstract}
  We study positive entropy birational automorphisms of threefolds. 
  We identify some conditions which imply that such an automorphism is non-regularizable.  
  We show that this criterion applies in the example of a positive entropy birational 
  automorphism of $\p^3$ constructed in \cite{Blanc_Pseudo}, thus showing that for a 
  general choice of parameters it is non-regularizable. Additionally, we establish a criterion 
  which proves that the automorphism in this example does not preserve a structure of a
  fibration over a surface.
\end{abstract}

 \maketitle

\section{Introduction}

Let $X$ be any normal projective variety defined over an algebraically closed field $k$ of characteristic~$0$.
We shall say that a birational automorphism $\varphi\colon X\dashrightarrow X$ is regularizable if 
there exists a birational 
map~\mbox{$\alpha\colon X\dashrightarrow Y$} to a variety $Y$ such that $\psi = \alpha\circ\varphi\circ\alpha^{-1}$ is a regular 
automorphism of $Y$. Understanding whether a given birational self-map is regularizable or not is a delicate problem.

% There are several cases when we can easily understand that the birational automorphism $\varphi\colon X\dashrightarrow X$
% of a projective variety $X$ is regularizable. 
We note immediately that any birational self-map of a curve, or of a variety of general type is regularizable. 
Any map  of finite order is also regularizable (see, e.g., \cite[ Lemma-Definition 3.1]{PS-regularization}).

% for instance, any birational automorphism of a curve is 
% in fact regular. As well any birational automorphism of a general type variety is regularizable by \cite{Cheltsov}.
% Also if the order of a birational automorphism is finite, then it is regularizable. Moreover, for a birational action of 
% any finite group on a variety it can be regularized on some birational model, see, e.g., \cite[ Lemma-Definition 3.1]{PS-regularization}.

The question whether one can regularize a birational automorphism $\varphi\colon X\dashrightarrow X$ of infinite order becames increasingly difficult
when dimension of $X$ grows. The case of surfaces is now well-understood and precise criteria have been devised in \cite{Diller_Favre} and \cite{Blanc-Cantat}
which are all based on the growth of degrees of the iterates of $\varphi$. 
Since the degree growth also plays a key role in all criteria that are known in higher dimension, we recall some basic properties of such.

Let us fix any ample line bundle $H$ over $X$. The $i$-th degree of $\varphi$ for $0\leqslant i\leqslant d$ is defined as the following number:
\begin{equation*}
 \deg_i(\varphi) = (\varphi)^*(H^i)\cdot H^{d - i}.
\end{equation*}
The growth of degrees $\deg_i(\varphi^n)$ is a birational invariant by \cite{Dinh_Sibony}. 
Moreover one can prove for all $0\leqslant i\leqslant d$, that the sequence $\deg_i(\varphi^n)$ for all $n>0$ is essentially submultiplicative;
thus, we can define the $i$-th dynamical degree of $\varphi$:
\begin{equation*}
 \lambda_i(\varphi) = \lim_{n\to \infty} \left(\deg_i(\varphi^n)\right)^{\frac{1}{n}}.
\end{equation*}
These numbers are birational invariants, they are real and satisfy $\lambda_i(\varphi)\geqslant 1$ for each $1\leqslant i\leqslant d-1$.
% If the growth of $\deg_i(\varphi)$ is polynomial, then $\lambda_i(\varphi)$ equals $1$. 
% If $\varphi$ is regular and $\lambda_1(\varphi)>1$, then by \cite{Gromov} and \cite{Yomdin} this automorphism has positive 
% topological entropy. 
% Following \cite{Gromov} and \cite{Yomdin} we say that $\varphi$ is a \emph{positive entropy} birational automorphism if $\lambda_1(\varphi)>1$. 
In case of birational automorphisms we have $\lambda_0(\varphi) = \lambda_d(\varphi) = 1$. 
The sequence of real numbers $\lambda_1(\varphi),\dots,\lambda_{d-1}(\varphi)$ is log-concave; i.e. we have the following inequalities
for all~\mbox{$1\leqslant i\leqslant d-1$:}
\begin{equation*}\label{eq_log_concavity}
 \lambda_{i-1}(\varphi)\cdot \lambda_{i+1}(\varphi)\leqslant \lambda_i(\varphi)^2.
\end{equation*}
In the sequel we shall use the following convenient terminology: $\varphi$ has positive entropy iff $\lambda_1(\varphi) >1$. 
Note that by log-concavity, the latter condition is equivalent to $\lambda_i(\varphi) >1$ 
for any $i = 1, \dots, d-1$.
% The dynamical degrees can be also defined for a K\"ahler manifold using a K\"ahler class instead of the ample divisors $H$.
For any birational automorphism $\varphi\colon X\dashrightarrow X$ of a variety $X$  
the growth of degrees $\deg_i(\varphi^n)$ and the dynamical degrees $\lambda_i(\varphi)$  give strong constraints on the possibility for~$\varphi$ to be regularizable.

We now list all currently known criteria ensuring regularizability or non-regularizability.
% Consider a birational automorphism $\varphi\colon X\dashrightarrow X$ of a projective surface $X$.
% Then the only non-trivial dynamical degree of $\varphi$ is $\lambda_1(\varphi)$ and we have the following possibilities:
\begin{itemize}
 \item If $X$ is a surface and $\lambda_1(\varphi) = 1$, then $\varphi$ is is regularizable if and only if $\deg_1(\varphi^n)$ is bounded 
 or~\mbox{$\deg_1(\varphi^n)\asymp n^2$} by \cite{Diller_Favre}. If  $\lambda_1(\varphi)>1$ is a Salem number, then $\varphi$ is regularizable by~\cite{Blanc-Cantat}.
 \item If $\deg_1(\varphi^n)$ is bounded and $X$ is any variety then $\varphi$ can be regularized by \cite{Weil_reg}.
 \item If $\deg_1(\varphi^n)\asymp n^d$ where $d$ is an odd integer and $X$ is any variety then $\varphi$ can not be regularized by~\cite[Proposition 1.2]{Cantat_Deserti_Xie}.

 \item If $X$ is any algebraic variety and $\varphi$ is regularizable, then $\lambda_1(\varphi)$ is an algebraic integer. 
%  This follows just from the fact that in this 
%  case $\lambda_1(\varphi)$ is the eigenvalue of the inverse image map on the Neron--Severi group of $X$. 
 However, if $\lambda_1(\varphi)$ lies in $\ZZ$
 and greater than $1$ then the birational automorphism~$\varphi$
 is not regularizable by  \cite[Proposition 1.1]{Cantat_Deserti_Xie}. In case when the dimension of $X$ equals 3 there are more restrictions 
 on the algebraic number $\lambda_1(\varphi)$ for a
 regularizable automorphism $\varphi$ which are listed in~\mbox{\cite[Proposition 4.6.7, 4.7.2, 5.0.1]{Lo_Bianco_kahler_threefolds}.} 
 
 \item If $X$ is a projective space $\p^d$ over a field $k$ of characteristic $0$ and $\varphi\colon \p^d \dashrightarrow \p^d$ is a birational automorphism, 
 then by \cite[Corollary 1.7]{Cantat_Deserti_Xie} one can find a linear transformation $A\in\PGL_{d+1}(k)$ such that $A\circ f$ is not regularizable. 
 Moreover, by \cite[Theorem 1.5]{Cantat_Deserti_Xie} the set of such $A$ is Zariski dense if $k$ is uncountable. Thus, a very general birational automorphism of $\p^d$ 
 over an uncountable field is not regularizable.
\end{itemize}

% The case of birational automorphism $\varphi\colon X\to X$ of a normal threefold $X$ is very hard and we know only few things about it. 
% First, if $\lambda_1(\varphi) = 1$ and $\deg(\varphi^n)\asymp n^d$ for some odd integer $d$, then $\varphi$ can not be regularized.
% If $\lambda_1(\varphi)>1$, then we have a numerical restriction, 
% namely, the first dynamical degree of a regularizable birational automorphism is an algebraic number with certain precisely described properties, 
% for more details see  \cite[Proposition 4.6.7, 4.7.2, 5.0.1]{Lo_Bianco_kahler_threefolds}. In particular, if $\lambda_1(\varphi)$ is an integer greater than~1, 
% than $\varphi$ can not be regularized, see also \cite[Proposition 1.1]{Cantat_Deserti_Xie}.

In this paper we explore the problem or regularization in the case when $X$ is a smooth  threefold and $\varphi\colon X\dashrightarrow X$ 
is a pseudo-automorphism such that~\mbox{$\lambda_1(\varphi)>1$}. 
Recall that a birational map is a pseudo-automorphism if $\varphi$ and $\varphi^{-1}$ does not contract any divisor in $X$. Note that in case of surfaces any 
pseudo-automorphism can be extended to a regular automorphism. Thus, pseudo-automorphisms form a class of birational automorphisms which are very 
close to being regular. One might expect that any pseudo-automorphism can be regularized; however, it turns out to be false. There are several examples of 
non-regularizable pseudo-automorphisms, e.g. \cite{Bedford-Kim_ex} and \cite{Bedford_Cantat_Kim}.
% 
% 
% 
% These kind of birational automorphisms is very convenient since in this case $\varphi$ 
% acts on the group of classes of divisors~$N^1(X)$
% as a linear automorphism and we can easily compute the number $\lambda_1(\varphi)$ as an eigenvalue of this action.

The easiest construction of an automorphism $\varphi$ of threefold $X$ with $\lambda_1(\varphi)>1$
is to take~\mbox{$X = S\times C$,} where $S$ is a surface admitting a positive entropy automorphism $f$ and $C$ is a curve. Then for 
any automorphism $g$ of $C$ the product $f\times g$ induces a positive entropy automorphism of $X$. 

However, we are mostly interested in 
automorphisms of varieties which can not be induced from lower dimensions.  To avoid such situations
we shall look at \emph{primitive} pseudo-automorphisms $\varphi\colon X \dashrightarrow X$, i.e. such that for any rational dominant map $\pi\colon X\to B$ 
to a variety $B$ of dimension $1\leqslant \dim(B)\leqslant \dim(X)- 1$ and any birational automorphism $f$ of $B$ we have inequality $\pi\circ\varphi \ne f\circ\pi$.

An interesting property \cite{Lesieutre_constraints} of a regular positive entropy automorphism $\varphi\colon X\to X$ of a smooth variety $X$ of dimension $3$ 
is that if $\delta\colon X\to Y$ is an extremal contraction, then either $\varphi$ induces a regular automorphism on $Y$ or $\varphi$ is not primitive.
This implies in particular that a positive entropy regular automorphism can not be constructed on any sequence of blow-ups with smooth centers of a 
smooth Fano variety. 
% That is why the construction of a regularization of a given positive entropy birational automorphism is a hard problem.

% In this paper we explore this problem in case of positive entropy pseudo-automorphisms of smooth projective threefolds
% which is arguably the first case to understand after the one of surfaces. More specifically we focus on pseudo-automorphisms $\varphi\colon X\dashrightarrow X$
% with $\lambda_1(\varphi)>1$, this condition allows us to consider vectors with special properties in $N_1(X)$. Moreover, we are interested in primitive
% pseudo-automorphisms, otherwise it is expected that the problem can be reduced to lower dimension.

We know several constructions of  positive entropy primitive pseudo-automorphisms of rational 
threefolds:~\mbox{\cite{Oguiso-Truong},  \cite{Bedford-Kim_Pseudo}, \cite{Bedford_Cantat_Kim},
\cite{Dolgachev_Ortland}, etc.}. The example in \cite{Oguiso-Truong} is the only 
known example of regular positive entropy primitive automorphism of a rational threefold. For examples  \cite{Bedford-Kim_Pseudo} and \cite{Bedford_Cantat_Kim} 
there was proved that they are not regularizable. In these papers authors proved that if $\varphi\colon X\dashrightarrow X$ is a 
birational automorphism and there is a $\varphi$-invariant surface $S$ in $X$ such that 
$\varphi|_S$ induces a non-regularizable birational automorphism of $S$, then $\varphi$ is also non-regularizable. 
Thus, they find an invariant  surfaces and applied criteria for positive entropy birational automorphisms for surfaces.

Our result is a new criterion of non-regularizibility of pseudo-automorphisms of threefolds.
% We consider pseudo-automorphisms of smooth threefolds which satisfies the following condition.

\begin{condition}\label{condition}
 We say that a pseudo-automorphism $\varphi\colon X\dashrightarrow X$ of a smooth projective threefold~$X$ satisfies Condition \ref{condition}
 if
 \begin{itemize}
  \item[{\textup{(1)}}] $\lambda_1(\varphi)^2> \lambda_2(\varphi);$
  \item[{\textup{(2)}}]  there exists a curve $C$ such that $\theta_1(\varphi)\cdot [C] <0;$
  \item[{\textup{(3)}}]  there exists infinitely many integers $m>0$ such that $C\not\subset \Ind(\varphi^{-m})$.
 \end{itemize}
\end{condition}
Some comments are in order. It is a theorem due to \cite{Truong} that the property $(1)$ implies the existence of a unique class $\theta_1(\varphi)$ 
in the group of classes of divisors $N^1(X)$
which satisfies the property $\varphi^*\theta_1(\varphi) = \lambda_1(\varphi)\theta_1(\varphi)$. Property $(2)$ implies that the class $\theta_1(\varphi)$ is 
not nef and the last property is required to avoid transformations obtained after flopping a curve with infinite orbit (see Section \ref{sect_examples}).

% The first property of Condition \ref{condition} is a slight modification of the log-concavity inequality.
% We need it by~\mbox{\cite[Theorem 1]{Truong}} to define the eigenclass $\theta_1(\varphi)$ in $N^1(X)$ with eigenvalue $\lambda_1(\varphi)$. 
% In case when $\varphi$ is regular the class $\theta_1(\varphi)$ is nef, the second property of Condition \ref{condition} implies that we are not in 
% such situation. 
% 
% Some justification of the third property of Condition \ref{condition} is given in Section \ref{subsec_ex_not-nef_reg}. 
% There we give a construction of a regularizable pseudo-automorphism $\varphi\colon X\dashrightarrow X$ of a threefold $X$ 
% which satisfies all but the last property of Condition~\ref{condition},
% i.e. there exists $\theta_1(\varphi)$-negative curve $C$ which lies in $\Ind(\varphi^{-m})$ for all $m>0$. 
% 
% 

% This condition allows us to formulate the following criterion:
\begin{theorem}\label{thm_criterion}
 Assume that $\varphi\colon X\dashrightarrow X$ is a pseudo-automorphism of a smooth projective threefold~$X$ 
 which satisfies Condition \textup{\ref{condition}}. Then $\varphi$ is not regularizable.
\end{theorem}
 The main step of the proof of Theorem \ref{thm_criterion} consists in proving that if we have a regularization $\psi\colon Y\to Y$ of $\varphi$ and if $\alpha\colon X\dashrightarrow Y$ is a 
 birational map such that $\varphi = \alpha^{-1}\circ\psi\circ\alpha$, then we have $\theta_1(\varphi) = \alpha^*\theta_1(\psi)$. This fact relies on our assumption that $\varphi$ 
 is a pseudo-automorphism, but it can fail in general. Property $(2)$ of Condition \ref{condition} implies that all proper images of $C$ are included in $\Ind(\alpha)$. 
 We then obtain a contradiction with property $(3)$. 

  % By theory developed in \cite{Dang_Favre_1} to any birational automorphism $\varphi\colon X\dashrightarrow X$ of a smooth
  % variety $X$ we can associate a $b$-divisor $\Theta_1(\varphi)$. If $\varphi$ is regularizable, then this class is
  % a Cartier $b$-divisor which is nef, i.e.~$\Theta_1(\varphi)$ is determined in some model $Y$ of $X$ and $\Theta_1(\varphi)|_Y$ is nef. 
  % In fact, Condition \ref{condition} describes a situation when the class $\Theta_1(\varphi)$ can not be nef $b$-Cartier divisor.

% In Section \ref{ex_nef_not-reg} we show that there exists an algebraically stable birational automorphism $\varphi\colon X\dashrightarrow X$ of a smooth threefold~$X$ 
% such that $\varphi$ is not regularizable, but the incarnation $\Theta_1(\varphi)|_X = \theta_1(\varphi)$ is nef.  
% This shows that Condition \ref{condition} is not necessary.
% However, we can see that in this example the class $\Theta_1(\varphi)$ is not a {Cartier} $b$-divisor. 
% Thus, we can assume that the right condition for a birational automorphism to be regularizable is that~$\Theta_1(\varphi)$ 
% is a nef Cartier $b$-divisors. 

Our second result is motivated by searching a criterion of primitivity for birational automorphisms.

\begin{theorem}\label{thm_primitivity}
 Assume that $\varphi\colon X\dashrightarrow X$ is pseudo-automorphism of a smooth projective threefold~$X$ such that $\lambda_1(\varphi)^2> \lambda_2(\varphi)$.
%  Then the following is true:
%  \begin{enumerate}
%   \item[{\textup{(1)}}] 
  If $\varphi$ satisfies Condition \textup{\ref{condition}}, then there exists no dominant rational map $\pi\colon X \dashrightarrow S$ 
  to a normal surface $S$ such that $\pi \circ \varphi = f \circ \pi$ for some birational map $f\colon S \dashrightarrow S$.
%   \item[{\textup{(2)}}] 
%   Suppose moreover that $X$ and $\varphi$ are defined over $\overline{\QQ}$ and that for any effective divisor $D$ on $X$ such that $\varphi^*D = D$  
%   there exists an irreducible $\theta_1(\varphi)$-negative curve $C$ in $X$ 
%   such that $D\cdot C\leqslant 0$. Then there exists no dominant rational map $\pi\colon X\dashrightarrow C$ to a normal curve $C$ and automorphism $f\colon C\to C$
%   such that $\pi\circ\varphi \ne f\circ\pi$.
%  \end{enumerate}
\end{theorem}
Condition \ref{condition} implies only that the pseudo-automorphism $\varphi$ can not  preserve the structure of a fibration over a surface.
The case of a fibration over a curve seems to be much harder.

% The second assertion of Theorem \ref{thm_primitivity} can only be applied to birational automorphisms defined over $\overline{\QQ}$. 
% However if we have a family of pseudo-automorphisms defined over $\CC$ such that
% for all points over $\overline{\QQ}$ we can apply this theorem, then by
% the semi-continuity theorem \cite[Theorem III.12.8]{Hartshorne} we can extend the result to a very general point of the family.

We apply Theorems \ref{thm_criterion} and \ref{thm_primitivity} to the family of pseudo-automorphisms introduced by J. Blanc in \cite{Blanc_Pseudo} that we now recall.
Let $Q\subset \p^{d}$ be a smooth cubic hypersurface for some $d\geqslant 2$. We associate a birational involution $\sigma_p$ of $\p^d$ to each point $p\in Q$.
For a general line $L$ passing through $p$ and intersecting $Q$ in three distinct points $p, q_1$ and $q_2$ we define $\sigma_p|_L$ as a unique non-trivial involution 
of $\p^1$ fixing both points $q_1$ and $q_2$. This defines a birational transformation of the projective space.

 Now take any general points $p_1,\dots,p_k$ on $Q$ with~\mbox{$k\geqslant 3$}.
 Then~\mbox{$\sigma_{p_1}\circ\dots\circ\sigma_{p_k}$} induces a positive
 entropy birational automorphism of $\p^d$ by \cite[Proposition 2.3]{Blanc_Pseudo}. 
 If $d = 2$, then the composition $\sigma_{p_1}\circ\dots\circ\sigma_{p_k}$ is a regularizable non-primitive 
 automorphism by \cite{Blanc_ex}.
 Here is the main result of this paper.

 \begin{theorem}
\label{theorem_blanc_example}
 Assume that $Q\subset \p^3$ is a very general smooth cubic surface over $\mathbb{C}$ and $p_1,p_2,p_3$ are general points on $Q$.
 Then the composition $\varphi = \sigma_{p_3}\circ\sigma_{p_2}\circ\sigma_{p_1}$
 is a positive entropy birational automorphism of $\p^3$ which is non-regularizable and does not preserve the structure of a fibration over a surface.
\end{theorem}
 Note that none of the previously known criteria from~\cite{Bedford-Kim_Pseudo} and \cite[Proposition 1.1]{Cantat_Deserti_Xie} apply in this situation.
%  However, using Theorems~\ref{thm_criterion} and \ref{thm_primitivity} we show that in some particular situation when $d = 3$ 
%  this automorphism is primitive and  non-regularizable.

To prove Theorem \ref{theorem_blanc_example} we find a $\theta_1(\varphi)$-negative curve $C$ on a birational model $X$ of $\p^3$ where $\varphi$ induces 
an algebraically stable automorphism.
The most difficult part in the proof of Theorem \ref{theorem_blanc_example} is to show that~$C$
satisfies property $(3)$ of Condition \ref{condition}. 
We do this by computing the orbit of a well-chosen point in $C$.
Namely, we take the cubic surface $Q$ such that its coefficients when we write it in some coordinates of $\p^3$ are algebraically independent.
We fix points $p_1,p_2$ and $p_3$, then we can write formulas defining $\sigma_i$ for $i = 1,2,3$.
Then we chose some concrete point in $\p^3$ and consider its images under the action of $\sigma_i$ as a set of four polynomials of coefficients of $Q$.
We show that after several iterations of involutions these polynomials has some form which remains the same after applying new involutions.
Thus, we manage to show that its orbit never falls into the indeterminacy locus of $\varphi$.
% Assuming that in these coordinates all coefficients of $Q$ are transcendental and algebraically independent we manage to show that 
% one concrete point on $C$ does not lie in the indeterminacy locus of~$\varphi^{-m}$ for all~$m>0$.

% That explains why we study only the case of three involutions: we use explicit formulas which are much more difficult for 4 or more involutions.
% However, it is a very interesting question: is it true that a general composition of more than three involutions can not be regularized.
Dealing with three involutions makes our computation already quite tricky. We believe that our theorem is valid for any composition 
of at least three involutions associated to general points on $Q$.

The paper is organized in the following way. In Section \ref{sect_preliminaries} we recall properties of birational maps.
In Section~\ref{sect_proofs} we prove Theorems \ref{thm_criterion} and \ref{thm_primitivity}.
In Section~\ref{sect_Blanc} we recall the construction of the positive entropy automorphism of $\p^3$ 
introduced in \cite{Blanc_Pseudo}, show that Condition \ref{condition} is satisfied for it  
and prove Theorem \ref{thm_blanc}. In Section \ref{sect_examples} we give an example of a 
regularizable pseudo-automorphism which satisfies all but third properties of Condition \ref{condition}.

\medskip

{\bf Acknowledgements.} I am very grateful to my advisor, Charles Favre, 
for his interest in this work and many useful suggestions.

\setcounter{tocdepth}{1}
\tableofcontents

\section{Preliminaries}
\label{sect_preliminaries}
\subsection{Birational maps acting on the divisor class group} Throughout this paper we consider smooth algebraic varieties over an algebraically closed field of characteristic 0. 
If~\mbox{$\alpha\colon X\dashrightarrow Y$} is a rational map between two varieties 
we denote by $\Ind(\alpha)$ the complement to the greatest open subset of $X$ on which $\alpha$ is regular, we call it the \emph{ locus of indeterminacy of~$\alpha$}.
By~$\Exc(\alpha)$ we denote the union of divisors in $X$ which are contracted under the action of $\alpha$ and we call it the \emph{exceptional locus of the map $\alpha$.}

Assume that $\alpha\colon X\dashrightarrow Y$ is a rational map between smooth varieties $X$ and $Y$. Consider a smooth variety~$V$ and two regular 
morphisms $\delta_X$ and~$\delta_Y$ to $X$ and $Y$ respectively such that the diagram
commutes and $\delta_X$ is birational:
  \begin{equation}\label{eq_birational_graph}
  \xymatrix{
  &V  \ar[ld]_{\delta_X} \ar[rd]^{\delta_Y} & \\
  X \ar@{-->}[rr]_{\alpha} && Y
  }
 \end{equation}
Note that such a diagram always exists (take, e.g., the resolution of indeterminacy of the graph of $\alpha$ in $X\times Y$).
Moreover, here we can assume that $\Ind(\alpha)$ coinsides with $\Ind(\delta_X^{-1})$.

By the \emph{total image} of a subset $W$ in $X$ we denote the subset $\alpha(W) = \delta_Y(\delta_X^{-1}(W))$ in $Y$. 
By the \emph{proper transform} of a subvariety $W$ which does not lie in $\Ind(\alpha)$ we denote the subvariety $\check{\alpha}(W) = \overline{\alpha(W\setminus \Ind(\alpha))}$ 
of~$Y$. Note that the proper transform of an irreducible subvariety is always irreducible.
There are many choices for the smooth variety $V$ as in the diagram \eqref{eq_birational_graph}. However, 
constructions of the proper transform and total image do not depend on this choice.
 
 If~$X$ is a smooth variety  we can define the $\RR$-vector 
space $N^1(X)$ of classes of Cartier divisors modulo numerical equivalence.  
For any element in the divisor class group $D\in N^1(Y)$ we can define its \emph{inverse image}:
\begin{align*}
 \alpha^*D = {\delta_X}_*(\delta_Y^*(D))\in N^1(X).
\end{align*}
This operation also does not depend on the choice of resolution of $\alpha$.
The divisor class group $N^1(X)$
of a smooth variety $X$ is a finite-dimensional vector space.
If $S$ is a proper irreducible reduced hypersurface in $X$, then we denote by $[S]$ the class of  $S$ in $N^1(X)$.

On a smooth variety $X$ we consider the group of numerical classes of curves which is the $\RR$-vector space $N_1(X)$ 
generated by classes of irreducible reduced curves modulo numerical equivalence.
If $C$ is an irreducible reduced curve in $X$, then we denote by $[C]$ its class in $N_1(X)$. 
On a smooth variety $X$ there is a natural perfect pairing between $N^1(X)$ and~$N_1(X)$.

We say that the class $D\in N^1(X)$ 
is \emph{nef} if $D\cdot [C]\geqslant 0$ for any effective curve $C$ on $X$.
The inverse image of a nef class under a regular map is nef.
In case of rational maps it is not true, but we have the following generalization of this fact.
\begin{lemma}\label{lemma_pb_nef_is_almost_nef}
 Consider a birational map $\alpha\colon X\dashrightarrow Y$ between smooth varieties $X$ and $Y$. If $D$ is a class of a nef divisor on $Y$ 
 and $C$ is a curve on $X$ such that $\alpha^*D\cdot [C] <0$, then $C$ lies in $\Ind(\alpha)$. In particular, in case~\mbox{$\dim(X) = 3$} the set 
 of $\alpha^*D$-negative curves is finite for any nef class of divisor $D$.
\end{lemma}
\begin{proof}
 Consider a diagram as in \eqref{eq_birational_graph}.
 If $D$ is a nef divisor in~$Y$, then its pullback $\widetilde{D} = \delta_Y^*D$ is also nef.
%  It suffices to show that if the class $\widetilde{D}$ on $V$ is nef, then $\delta_{X*}(\widetilde{D})$ is a nef class on $X$.
 
 Consider the class $\delta_X^*(\delta_{X*}(\widetilde{D}))$; the difference $E$ between this class and $\widetilde{D}$ is supported
 on the exceptional locus of $\delta_X$:
 \begin{equation*}
  \widetilde{D} = \delta_X^*(\delta_{X*}(\widetilde{D})) - E.
 \end{equation*}
 By Lemma \ref{lemma_composition_of_pb} we get that $E$ is the class of an effective divisor. 
 
 Take any irreducible curve $C$ on $X$ outside $\delta_X(\Exc(\delta_X))$. Denote by $\tC$ the proper transform of $C$ in~$V$ so 
 that $\delta_{X*}[\tC] = [C]$. Then we have
 \begin{equation*}
  \delta_{X*}(\widetilde{D})\cdot [C] = \delta_{X*}(\widetilde{D})\cdot \delta_{X*}[\tC] = \delta_X^*(\delta_{X*}(\widetilde{D}))\cdot [\tC] = 
  (\widetilde{D} + E)\cdot [\tC]\geqslant 0.
 \end{equation*}
 The last inequality is true since $\widetilde{D}$ is nef, $E$ is effective and $\tC$ lies outside the support of $E$. Thus, we obtain 
 that the product of $\delta_{X*}(\widetilde{D})$  and any curve outside $\delta_X(\Exc(\delta_X))$ is not negative and this implies the result. 
\end{proof}

% \begin{lemma}
%  Assume that $\alpha\colon X \dashrightarrow Y$ is a rational map beetween normal varieties and $S$ is an irreducible surface on $Y$.
%  Then we have $\alpha^*[S] = 0$ if for any resolution of the graph of $\alpha$ in \eqref{eq_birational_graph} all divisors in $\delta_Y^{-1}(S)$ are contracted by $\delta_X$.
% \end{lemma}
% \begin{proof}
%  
%  
% \end{proof}

By the above definition a rational map $\alpha\colon Y\dashrightarrow X$  between smooth varieties $X$ and $Y$ defines 
a map between their groups of numerical classes of divisors $ \alpha^*\colon N^1(X)\to N^1(Y)$. If $\alpha$ is regular, then $\alpha^*$ is 
the standard functor of the inverse image; in particular, the composition of the inverse images of two 
regular maps is equal to the inverse image of the composition of maps. In case of rational maps the situation is more complicated.
\begin{lemma}\label{lemma_composition_of_pb}
 Let $X, Y$ and $Z$ be smooth varieties and $\alpha\colon Y\dashrightarrow X$ and $\beta\colon Z\dashrightarrow Y$ be rational maps such that $\beta(Z)$ does not lie in $\Ind(\alpha)$. 
 Then the composition $\alpha\circ\beta\colon Z\dashrightarrow X$ is well-defined and for each class~\mbox{$D\in N^1(X)$} we have
 the following equality:
 \begin{equation}\label{eq_difference_of_pb}
  \beta^*(\alpha^*(D)) - (\alpha\circ\beta)^*(D) = E,
 \end{equation}
  where $E$ is a divisor on $Z$ supported in $\beta^{-1}(\Ind(\alpha))$. If $D$ is nef, then $E$ is effective.
\end{lemma}
\begin{proof}
 Denote by $\Gamma_{\alpha}$ and $\Gamma_{\beta}$ resolutions of singularities of graphs of maps $\alpha$ and $\beta$ respectively.
 Denote by~$p_{\alpha}$, $q_{\alpha}$ projections to $X$ and $Y$ from $\Gamma_{\alpha}$ and by $p_{\beta}$ and $q_{\beta}$
 projection from $\Gamma_{\beta}$ to $Y$ and $Z$. Both maps $q_{\alpha}$ and $q_{\beta}$ are birational and $q_{\alpha}(\Exc(q_{\alpha}^{-1})) = \Ind(\alpha)$. 
 The composition $q_{\alpha}^{-1}\circ p_{\beta}$ induces the rational map  $\gamma \colon \Gamma_{\beta} \dashrightarrow \Gamma_{\alpha}$.
 Denote by $\Gamma$ the resolution of the graph of $\gamma$
 and by $p$ and $q$ projections to $\Gamma_{\alpha}$ and $\Gamma_{\beta}$. The morphism $q$ is birational and the following diagram commutes:
 \begin{equation*}
  \xymatrix{
  &&\Gamma \ar[ld]_{q} \ar[rd]^p &&\\
  & \Gamma_{\beta} \ar[ld]_{q_{\beta}} \ar[rd]^{p_{\beta}} \ar@{-->}[rr]^{\gamma} && \Gamma_{\alpha} \ar[ld]_{q_{\alpha}} \ar[rd]^{p_{\alpha}} &\\
  Z \ar@{-->}[rr]^{\beta} && Y \ar@{-->}[rr]^{\alpha} && X
  }
 \end{equation*}
 By definition $(\alpha \circ\beta)^*D = q_{\beta *}(q_*(p^*(p_{\alpha}^*D)))$ and $\beta^*(\alpha^*D) = q_{\beta *}(p_{\beta}^*(q_{\alpha *}(p_{\alpha}^*D)))$.
 Denote by $\widetilde{D}$ the class $p_{\alpha}^*D$ in $N^1(\Gamma_{\alpha})$ and denote by $\widetilde{E}$ the following difference:
 \begin{equation*}
  \widetilde{E} = p_{\beta}^*(q_{\alpha *}\widetilde{D}) - q_*(p^*\widetilde{D}).
 \end{equation*}
 By definition we see that $\beta^*(\alpha^*D) - (\alpha\circ\beta)^*D= q_{\beta *}\widetilde{E}$. 
 Consider the following equality:
 \begin{multline*}
  \widetilde{E} = q_*(q^*\widetilde{E}) = q_*(q^*(p_{\beta}^*(q_{\alpha *}\widetilde{D}))) - q_*(p^*\widetilde{D}) =\\
  = q_*(p^*(q_{\alpha}^* (q_{\alpha *}\widetilde{D}))) - q_*(p^*\widetilde{D}) = q_*(p^*(\widetilde{D} + E')) - q_*(p^*\widetilde{D}) =  q_*(p^* E'),
 \end{multline*}
 where $E' = q_{\alpha}^* (q_{\alpha *}\widetilde{D}) - \widetilde{D}$ is a  class of divisor with the support in $\Exc(q_{\alpha})$. 
 Thus, the image of the support of $\widetilde{E} = q_*(p^* E')$ under $p_{\beta}$ lies in $\Ind(\alpha)$.
 This implies that $E$ is a divisor on $Z$ supported in $\beta^{-1}(\Ind(\alpha))$.
 
 Now assume that $D$ is a nef divisor. Then so is $\widetilde{D}$; thus, the class $E'$ is $q_{\alpha}$-antinef 
 by construction and~$q_{\alpha *}E' = 0$ is an effective divisor.
 By \cite[Lemma 3.39]{Kollar_Mori} we get that $E'$ is effective. 
 Consider the divisor class $q^*E'$; it is $q$-nef and $E' = q_*(q^*E')$ is effective. Thus, by \cite[Lemma 3.39]{Kollar_Mori} we get that $q^*E'$ is effective.
 Then so is $E$ and this finishes the proof.
\end{proof}

Let $\alpha\colon X\dashrightarrow Y$ be a rational map between smooth varieties $X$ and $Y$. 
Then using the natural perfect pairing between $N^1(X)$ and $N_1(X)$ we can define the direct 
image $\alpha_*\colon N_1(X)\to N_1(Y)$ as the dual map to $\alpha^*\colon N^1(Y)\to N^1(X)$.
If the curve $C$ does not lie in the indeterminacy locus of $\alpha$, then we have the following interpretation of the direct image of a curve:
\begin{lemma}\label{lemma_pb_of_C_under_bu}
 If $\delta\colon Y\to X$ is a blow-up of a smooth variety $X$ in smooth center $Z$ and $C \not\subset Z$ is an irreducible curve in $X$, 
 then for any irreducible curve $T$ such that $\delta(T)$ is a point we have 
 \begin{equation*}
  \delta^*[C] = [\widetilde{C}] + \mu[T]\in N_1(X),
 \end{equation*}
 where $\widetilde{C}$ is the proper transform of $C$ and $\mu$ is a non-negative number.
\end{lemma}
\begin{proof}
  Take some class of divisor $D$ in $X$. Since 
 \begin{equation*}
   \delta^*D\cdot ([\widetilde{C}] - \delta^*[C])  = D\cdot \delta_{*}[\widetilde{C}] - \delta^*(D\cdot[C])  = 0,
 \end{equation*}
 we get that $\delta^*[C] = [\widetilde{C}] + \xi$, where $\xi\in N_1(X)$ is a class of curve such that $\delta_{1*}(\xi) = 0$. 
 Denote by $E$ the exceptional divisor of $\delta$.
 Since $\delta$ is a blow-up along a smooth center, the class $\xi$ is proportional to the class of the curve $T$ in $E_1$ such that $\delta(T)$ is a point.
 By the projection formula we have
 \begin{equation*}
   0 = E\cdot \delta^*[C] = E \cdot ([\widetilde{C}] + \xi) = E\cdot [\widetilde{C}] + E \cdot \xi
 \end{equation*}
 Since $\widetilde{C}$ does not lie inside $E$ the product $E\cdot [\widetilde{C}]$ is non-negative. Thus, $E\cdot \xi \leqslant 0$
 then we get that~\mbox{$\xi = \mu [T]$} is effective and  $\mu\geqslant 0$.
\end{proof}

\begin{lemma}\label{lemma_pf_of_curve}
 Assume that $\alpha\colon X\dashrightarrow Y$ is a birational map between smooth varieties $X$ and $Y$ and $C$ is an irreducible curve on $X$ such that $C\not\subset\Ind(\alpha)$.
 If $\widetilde{C}$ is the proper image of $C$ under $\alpha$, then there exist effective curves  $T_1,\dots,T_N$ in $\alpha(\Ind(\alpha))$ and non-negative numbers 
 $\mu_1,\dots \mu_N \geqslant 0$ such that
 \begin{equation*}
  \alpha_*[C] = [\widetilde{C}] + \sum_{i=1}^M \mu_i [T_i] \in N_1(X).
 \end{equation*}
  In particular, if $C$ does not lie in $\Ind(\alpha)$, then the class $\alpha_*[C]$ is effective.
\end{lemma}
\begin{proof}
 Consider the diagram as in \eqref{eq_birational_graph}. We can assume that $\delta_X = \delta_N\circ\dots\circ \delta_1$ 
 is a composition of  several blow-ups~\mbox{$\delta_i\colon X_i\to X_{i-1}$} along smooth centers with $X = X_0$ and $V = X_N$.
 By definition we have~\mbox{$\alpha_*[C] = \delta_{Y*}(\delta_X^*[C])$.}
 
 Consider the blow-up $\delta_i\colon  X_i\to X_{i-1}$ along a smooth center $Z$ and a curve $T$ which does not lie in $Z$.
 Denote by $\widetilde{T}$ the proper preimage of $T$ in $X_i$. Denote by $E$ the exceptional divisor of $\delta_i$ and 
 consider the intersection number $n =E\cdot[\widetilde{T}]$. It is positive since $\widetilde{T}$ does not lie in $E$.
 Then by projection formula we get $\delta_i^*[T] = [\widetilde{T}] + n\cdot F$, where $F$ is the class of the curve on the extremal ray of the contraction $\delta$.
 The class $F$ is effective, its representative lies in $E$ and can be choosen such that it does not lie in a closed set of codimension~$2$.
 
 By induction we prove that  $\delta_X^*[C] = \check{\delta}_X(C) + \sum\mu'_i [T'_i]$, where $\check{\delta}_X(C)$ is the proper
 transform of $C$ under~$\delta_X$, curves $T'_i$ lie in $\Exc(\delta_X)$ and numbers $\mu'_i$ are non-negative. 
 
 Now by \cite{Fulton_IT} the direct image under a regular map $\delta_{Y*}$ of the linear combination of classes of 
 effective curves $\delta_X^*[C] = \check{\delta}_X(C) + \sum\mu'_i [T'_i]$ is a linear combination with the same coefficients of 
 images of these curves under $\delta_Y$. This finishes the proof.
\end{proof}

 Recall also
that $\varphi$ is \emph{regularizable},
if there exist a smooth variety $Y$, a birational map $\alpha$ and a regular automorphism $\psi\colon Y\to Y$ such that the following diagram commutes: 
\begin{equation*}\label{eq_regularization}
 \xymatrix{
  X\ar@{-->}[d]_{\alpha}\ar@{-->}[r]^{\varphi}&X\ar@{-->}[d]_{\alpha}\\Y \ar[r]^{\psi}  & Y 
 }
\end{equation*}
In this situation we call the triple $(Y,\psi, \alpha)$ is a smooth regularization of $\varphi$. 
By functorial desingularization (see, for instance, \cite[Theorem 3.26]{Kollar_resolution}) any
birational map which admits a regularization also admits
a smooth regularization.

 Recall that a birational map $\varphi\colon X\dashrightarrow Y$ is called a \emph{pseudo-isomorphism}, if sets $\Exc(\varphi)$ and $\Exc(\varphi^{-1})$ 
 are empty. If $Y = X$, then we call such map a \emph{pseudo-automorphism}. 
 On surfaces the notion of pseudo-automorphisms coincides with the one of regular automorphisms, 
 while in higher dimensions the notions differ. Nevertheless, in higher dimension there are restrictions on the indeterminacy locus of
 a pseudo-automorphism.
 \begin{lemma}[{\cite{Bedford-Kim_Pseudo}}] \label{lemma_pseudo-aut_ind_consists_of_curves}
  If $\varphi\colon X\dashrightarrow X$ is a pseudo-automorphism of a smooth variety $X$ of dimension $3$ or greater,
  then $\Ind(\varphi)$ has no isolated points.
 \end{lemma} 
 Observe that pseudo-automorphisms  induce invertible maps on the group of classes of divisors on the variety:
 \begin{lemma}\label{lemma_pseudo_aut}
  If $\varphi\colon X\dashrightarrow X$ is a pseudo-automorphism of a smooth variety $X$, then
  \begin{equation*}\label{eq_map_on_NS}
   \varphi^*\colon N^1(X)\to N^1(X)
  \end{equation*}
  is an isomorphism of vector spaces and $(\varphi^n)^* = (\varphi^*)^n$.
 \end{lemma}
 \begin{proof}
  This follows from Lemma \ref{lemma_composition_of_pb}.
 \end{proof}

\subsection{A special construction of a flop}
In this section we consider the construction of a concrete pseudo-isomorphism and construct the resolution of its graph.

Consider a smooth threefold $X$ and two smooth curves $\Gamma_1$ and $\Gamma_2$ on $X$ such that the intersection $\Gamma_1\cap\Gamma_2$
is a finite set of points and the union $\Gamma_1\cup\Gamma_2$ is a nodal curve.

 Denote by $\delta_1\colon Y_1\to X$ the blow-up of $X$ along the curve $\Gamma_1$. 
 Let $\tG_2\subset Y_1$ be the proper transform of the curve $\Gamma_2$ under $\delta_1$.
 Let $\delta_{12}\colon Y_{12}\to Y_1$ be the blow-up of $Y_1$ along $\tG_2$.
 
 Denote by $\delta_{2+}\colon Y_{2+} \to X$ the blow-up of $X$ along the curve $\Gamma_2$, by $\tG_1$ the proper transform of $\Gamma_1$ to~$Y_{2+}$ and 
 by $\delta_{12+}\colon Y_{21+} \to Y_{2+}$ the blow-up of $Y_{2+}$ in $\tG_1$. Consider the birational map $\tau\colon Y_{12}\to Y_{12+}$ induced by the following diagram:
 \begin{equation*}
  \xymatrix{ Y_{12}\ar[d]_{\delta_{12}}\ar@{-->}[rr]^{\tau} && Y_{21+} \ar[d]^{\delta_{21+}} \\ 
             Y_1\ar[dr]_{\delta_1} && Y_{2+} \ar[ld]^{\delta_{2+}} \\ &X&}
 \end{equation*}
 The map $\tau$  is a pseudo-isomorphism and its restriction to 
 $(\delta_{1}\circ\delta_{12})^{-1}\left(X\setminus \left(\Gamma_1\cap\Gamma_2\right)\right)\subset Y_{12}$ induces an isomorphism 
 to $(\delta_{2+}\circ\delta_{21+})^{-1}\left(X\setminus \left(\Gamma_1\cap\Gamma_2\right)\right)\subset Y_{21+}$.

 For each point $p_i$ in $\Gamma_1\cap\Gamma_2$ we consider a curve $C_i$ defined as the irreducible component 
of $\left(\delta_{1}\circ\delta_{12}\right)^{-1}(p_i)$ which does not lie in the exceptional divisor of $\delta_{12}$.
Analogously, we denote by $C_i^+$ the irreducible component of $\left(\delta_{2+}\circ\delta_{21+}\right)^{-1}(p_i)$
which does not lie in the exceptional divisor of $\delta_{21+}$. These are smooth rational curves 
and curves $C_i$ and $C_j$ (respectively, $C_{i}^+$ and $C_j^+$) do not intersect 
in $Y_{12}$ (respectively $Y_{21+}$) for distinct $i$ and $j$.

Then the map $\tau$ is a pseudo-isomorphism between $Y_{12}$ and $Y_{21+}$; it is a flop, see \cite[Example 1.12]{Flips_and_flops}.
Denote by $f_{1}$ and $f_{2}$ the curve classes of the fibers of the morphisms $\delta_1$ and $\delta_{12}$ over $\Gamma_1$ and $\tG_2$ respectively. Then the following assertion is true:
\begin{proposition}[{\cite[Example 1.12]{Flips_and_flops}}]\label{prop_flop}
 The map $\tau$ is a pseudo-isomorphism.
 The sets of indeterminacy of maps~$\tau$ and~$\tau^{-1}$ consist of disjoint unions of curves:
 \begin{align*}
  \Ind(\tau) = C_1\sqcup C_2 \sqcup \dots \sqcup C_m; &&   \Ind(\tau^{-1}) = C^+_1\sqcup C_2^+ \sqcup \dots \sqcup C^+_m. 
 \end{align*}
 Moreover, the class of each curve $C_i$ and of its direct image under $\tau$ satisfy:
 \begin{align*}
  [C_i] = f_1 - f_2 && \tau_*[C_i] = -[C_i^+].
 \end{align*}
\end{proposition}
\begin{proof}[Sketch of proof]
 Denote by $\Delta\colon V\to Y_{12}$ the blow-up of the variety $Y_{12}$ in the smooth curve $C_1\sqcup\dots\sqcup C_m$ and 
 by~\mbox{$\Delta_+\colon V\dashrightarrow Y_{21+}$} the induced map to $Y_{21+}$. 
 
 We use the universal property of blow-ups (see \cite[Proposition II.7.14]{Hartshorne}).
%  if $f\colon A\to B$ is a regular map between smooth varieties of any dimension and the preimage $f^{-1}(Z)$ of some subvariety $Z\subset B$ is 
%  a Cartier divisor in $A$, then $f$ factors through the blow-up~\mbox{$\delta\colon \Bl_Z(B) \to B$}; i.e. there exists a map $f'\colon A \to \Bl_Z(B)$ such that 
%  the following diagram commutes:
%  \begin{equation*}
%   \xymatrix{
%   A\ar[dr]_f \ar[r]^{f'} & \Bl_Z(B) \ar[d]^{\delta} \\ & B
%   }
%  \end{equation*}
 First we apply this property to the map $\delta_1\circ\delta_{12}\circ\Delta \colon V\to X$. Since the preimage of the curve $\Gamma_2$ 
 is of pure dimension $2$ and since $V$ is smooth we get that the following map is regular:
 \begin{equation*}
  \left(\delta_{2+}^{-1}\circ\delta_1\circ\delta_{12}\circ\Delta\right) \colon V\to Y_{2+}.
 \end{equation*}
 Repeating this argument we can show that there exists a regular map $ f\colon V \to V_+$, 
 where $V_+$ is a blow-up of the variety $Y_{21+}$ in the smooth curve $C_{1+}\sqcup\dots\sqcup C_{m+}$ such that the following diagram commutes:
 \begin{equation*}
  \xymatrix{ & V \ar[ld]_{\Delta} \ar[r]^{f} \ar[rrd]_{\Delta_+} & V_+ \ar[rd] & \\
   Y_{12} \ar[rrr]^{\tau} &&& Y_{21+}}
 \end{equation*}
 Since Picard numbers of $V$ and $V_+$ are same, this implies that the map $f$ is an isomorphism. 
 
 Considering the Picard group of $Y_{12}$ we can prove that $[C_i] = f_1 - f_2$. 
 Note that since $\tau$ is a pseudo-isomorphism, groups $N^i(Y_{12})$ and $N^i(Y_{21+})$ are isomorphic for $i = 1$. 
 Since $Y_{12}$ is a threefold and group of classes of curves is dual to the divisor class group, we have also the isomorphism for $i = 2$.
 Under this identification we get $[C_{i}^+ ]= f_2 - f_1$, thus, $\tau_*[C_i] = -[C_{i}^+ ]$.
 \end{proof}

\subsection{Dynamical degrees} 
If $\varphi$ is any birational automorphism of a smooth projective variety $X$,
 then we define its dynamical degrees $\lambda_i(\varphi)$ for all $0\leqslant i\leqslant \dim(X)$ as follows:
  \begin{equation*}
   \lambda_i(\varphi) = \lim_{n\to\infty} \left((\varphi^n)^*(H^i)\cdot H^{\dim(X) - i}\right)^\frac{1}{n}, 
  \end{equation*}
 where $H$ is any ample class in $N^1(X)$. The fact that the limit exists and does not depend on the choice of the ample class 
 is proved in \cite{Dinh_Sibony}, \cite{Truong_rel_deg}.
 Dynamical degrees are positive real numbers which are greater than or equal to $1$ and they are birational invariants of the automorphism $\varphi$. 
 The projection formula implies:
 \begin{equation*}
  \lambda_i(\varphi) = \lambda_{\dim(X)-i}(\varphi^{-1}).
 \end{equation*}  
%  In case of regularizable automorphisms the positivity of the algebraic entropy is equivalent to the
%  positivity of the topological entropy by \cite{Gromov} and \cite{Yomdin}.   
 Dynamical degrees are log-concave, see \cite{Dinh_Nguyen}; i.e. for all indices $0\leqslant i\leqslant \dim(X)$ one has:
  \begin{equation*}
   \lambda_i(\varphi)^2\geqslant \lambda_{i+1}(\varphi)\lambda_{i-1}(\varphi).
  \end{equation*}  
%  We set $h_{\alg} = \max_{0\leqslant i\leqslant \dim(X)}(\lambda_i(\varphi)$. 
 Note that by log-concavity we have $\lambda_1(\varphi) = 1$ if and only if $\lambda_i(\varphi) = 1$ for all $1\leqslant i\leqslant \dim(X)$. 
 By log-concavity we have also the following inequality $\lambda_1(\varphi)^2\geqslant\lambda_2(\varphi)$ for all birational automorphisms~$\varphi$. 
%  When~$\varphi$
%  is a pseudo-automorphism note that Lemma \ref{lemma_pseudo_aut} immediately implies that $\lambda_1(\varphi)$ equals to the greatest eigenvalue 
%  of $\varphi^*\colon N^1(X) \to N^1(X)$. 
 If this inequality is strict and $\varphi$ is a pseudo-automorphism, then the action of $\varphi^*$ on the group of classes of divisors has the following property:

 \begin{theorem}[{\cite[Theorem 1, Corollary 3]{Truong}}]\label{thm_Truong_condition}
  Assume that $\varphi\colon X\dashrightarrow X$ is a pseudo-automorphism of a smooth projective variety $X$ satisfying $\lambda_1^2(\varphi)>\lambda_2(\varphi)$.
  Then there exists a non-zero class~\mbox{$\theta_1(\varphi)\in N^1(X)$} such that:
  \begin{enumerate}
   \item[\textup{(1)}]  For any ample class $H$ the limit $\lim_{n\to\infty} \frac{(\varphi^n)^*( H)}{\lambda_1(\varphi)^n}$ exists, is non-zero and 
   proportional to $\theta_1(\varphi);$
   \item[\textup{(2)}] $\varphi^*(\theta_1(\varphi)) = \lambda_1(\varphi)\theta_1(\varphi);$
   \item[\textup{(3)}] The eigenvalue $\lambda_1(\varphi)$ is simple, i.e. there is a $\varphi^*$-invariant decomposition $N_1(X) = \RR \theta_1(\varphi)\oplus V$ .
  \end{enumerate}
 Moreover, the absolute value of any eigenvalue of $\varphi^*$ distinct from $\lambda_1(\varphi)$ 
 is less than or equal to $\sqrt{\lambda_2(\varphi)}$
 \end{theorem}
 In \cite{Dang_Favre_1} was proved a generalization of this theorem in the case of any birational automorphism. 
%  In this situation under the condition $\lambda_1^2(\varphi)>\lambda_2(\varphi)$  we cannot construct the eigenvector $\theta_1(\varphi)$ with eigenvalue~$\lambda_1(\varphi)$.
%  However, we can consider the group of $b$-divisors on $X$, it is an infinitely generated group. Choosing some special norm on the group of $b$-divisors
%  we can construct a $b$-divisor analogous to the class $\theta_1(\varphi)$ and prove properties $(1 - 3)$ in this setting.

%  In the assumption of Theorem \ref{thm_Truong_condition} the class $\theta_1(\varphi)$ is movable. 
%  Moreover, if $\varphi$ is regular or if $\dim(X) = 2$, then the class $\theta_1(\varphi)$ is nef. 
%  However, even in case of a pseudo-automorphisms of threefolds this class can intersect some curves negatively. 
 
 In case when the birational automorphism $\varphi$ is regular or if $\dim(X) = 2$, then the class $\theta_1(\varphi)$ is nef. 
 However, even in the case of a pseudo-automorphisms of threefolds this class can intersect some curves negatively.
 
 \begin{lemma}\label{lemma_theta-negative_in_ind}
 Let $\varphi\colon X\dashrightarrow X$ be a pseudo-automorphism of a smooth threefold $X$ 
 such that~\mbox{$\lambda_1(\varphi)^2>\lambda_2(\varphi)$.}
 If $C$ is an irreducible curve and $\theta_1(\varphi)\cdot [C]<0$, then there exists an integer  $N$ such that
 \begin{equation*}
  C\subset\bigcap_{n>N}\Ind(\varphi^n).
 \end{equation*}
\end{lemma}
\begin{proof}
 Fix a curve $C$ such that $\theta_1(\varphi)\cdot C <0$ and an ample divisor $H$ on $X$. 
 Then by Theorem \ref{thm_Truong_condition} there exists $N$ such that for all $n>N$ we have the following inequality
 \begin{equation*}
  f^{n*}(H) \cdot C <0.
 \end{equation*}
 By Lemma \ref{lemma_pb_nef_is_almost_nef} this is possible only if $C$ lies in $\Ind(\varphi^n)$ for all $n>N$.
\end{proof}
The set of curves $C$ such that $\theta_1(\varphi)\cdot [C]<0$ is finite hence $\theta_1(\varphi)$ is movable in the sense of \cite{Movable_cone}.

\section{Proofs of Theorems \ref{thm_criterion} and \ref{thm_primitivity}}
\label{sect_proofs}
\subsection{Regularizations of pseudo-automorphisms}
% Consider a pseudo-automorphism $\varphi\colon X \dashrightarrow X$ such that $\lambda_1(\varphi)^2 >\lambda_2(\varphi)$.
% By Theorem \ref{thm_Truong_condition} there exists an eigenclass $\theta_1(\varphi)\in N^1(X)$ with eigenvalue $\lambda_1(\varphi)$ such 
% that 
% \begin{equation*}
%  \theta_1(\varphi) = M_H\cdot\lim_{n\to \infty} \frac{\varphi^{n*}H}{\lambda_1(\varphi)^n},
% \end{equation*}
% where $H$  is any ample divisor on $X$ and $M_H$ is a strictly positive constant depending only on $H$.
% Then any~\mbox{$\theta_1(\varphi)$-ne}\-gative curve lies in the indeterminacy locus of some power of $\varphi$.
%  
% 
% In particular, Lemma \ref{lemma_theta-negative_in_ind} implies the first assertion of Theorem \ref{thm_criterion}.

Here we consider smooth threefolds~$X$ and~$Y$. Let  $\varphi$ be a pseudo-automorphism of $X$ and $\psi$ be a regular automorphism of $Y$.
The birational map $\alpha\colon X\dashrightarrow Y$ is such that $\varphi\circ\alpha = \alpha\circ\psi$.

Consider some resolution of the graph of the map $\alpha$:
\begin{equation}\label{diag_elimination}
   \xymatrix{
  &V  \ar[ld]_{\delta_X} \ar[rd]^{\delta_Y} & \\
  X\ar@{-->}@(dl,ul)^{\varphi} \ar@{-->}[rr]_{\alpha} && Y \ar@(dr,ur)_{\psi}
  }
\end{equation}
We can choose $V$, $\delta_X$ and $\delta_Y$ in  such a way that $\Ind(\alpha) = \delta_X(\Exc(\delta_X))$. Moreover, we can assume that $\delta_X$ 
is a composition of blow-ups along smooth centers.

We consider classes of divisors~$\theta_1(\varphi)$ and $\theta_1(\psi)$ in $N^1(X)$ and $N^1(Y)$ respectively;
inverse images of these classes to $V$ are connected in the following way:
\begin{lemma}\label{lemma_pb_of_theta}
 Let $\varphi$ be a pseudo-automorphism with a regularization $(Y,\psi,\alpha)$ fitting into a diagram of the form \textup{\eqref{diag_elimination}}.
 If $\lambda_1(\varphi)^2>\lambda_2(\varphi)$, then there exists a class $E$ of an effective divisor in $\Exc(\delta_X)$ such that
 \begin{equation*}
  \delta_X^*\theta_1(\varphi) = \delta_Y^*\theta_1(\psi) +E.
 \end{equation*}
\end{lemma}
\begin{proof}
 First we justify that $\alpha^*\theta_1(\psi) \neq 0$. Suppose by contradiction 
 that~\mbox{$\alpha^*\theta_1(\psi) = \delta_{X*}(\delta_Y^*(\theta_1(\psi)) = 0$}. Since $\delta_X$ is a composition of blow-ups along smooth centers this implies that 
 the divisor class $\delta_Y^*(\theta_1(\psi)$ is supported on a subset of $\Exc(\delta_X)$.
 then there exists a class $E$ of the divisor on $V$ supported on $\Exc(\delta_X)$ such that
 \begin{equation*}
  \delta_Y^* \theta_1(\psi) = E.
 \end{equation*}
 The class $\theta_1(\psi)$ is nef, then so is $E$. By Mori negativity lemma \cite[Lemma 3.39]{Kollar_Mori} since $\delta_{X*}E = 0$ is effective
 we conclude that $E$ is anti-effective.
 However, $\delta_{Y*} E = \theta_1(\psi)$ is a pseudo-effective class. Thus, we get a contradiction.

\medskip
 Now we can assume that $\alpha^*\theta_1(\psi)$ is a non-zero class. We apply Lemma~\ref{lemma_composition_of_pb}:
 \begin{equation*}
  \lambda_1(\psi)\alpha^*\theta_1(\psi) = \alpha^*(\psi^*\theta_1(\psi)) = (\psi\circ\alpha)^*\theta_1(\psi) = 
  (\alpha\circ\varphi)^*\theta_1(\psi) = \varphi^*(\alpha^*\theta_1(\psi)).
 \end{equation*}
 By \cite[Theorem 1.1]{Truong_rel_deg}  we have $\lambda_1(\psi) = \lambda_1(\varphi)$. 
 Since both classes $\theta_1(\psi)$ and $\theta_1(\varphi)$ are pseudo-effective, there exists a number $C>0$ such that
 \begin{equation*}
 \theta_1(\varphi) = C\cdot\alpha^*\theta_1(\psi).  
 \end{equation*}
   Pulling back this equation by $\delta_X$ we obtain an effective divisor $E$ supported in $\Exc(\delta_X)$ such that
 \begin{equation*}
  \delta_X^*\theta_1(\varphi) = C\cdot\delta_Y^*\theta_1(\psi)+ E.
 \end{equation*}
 Since the class $\theta_1(\psi)$ is nef, the class $-E$ is $\delta_X$-nef. Then
 by \cite[Lemma 3.39]{Kollar_Mori} we get that $E$ is effective.
\end{proof}

\begin{lemma}\label{lemma_theta-negative_curves}
 Let $\varphi$ be a pseudo-automorphism with a regularization $(Y,\psi,\alpha)$ fitting into a diagram of the form \textup{\eqref{diag_elimination}}
 and $\lambda_1(\varphi)^2>\lambda_2(\varphi)$.
 Then if~$C$ is an irreducible curve such that $\theta_1(\varphi)\cdot [C]<0$, then $C$ lies in $\Ind(\alpha)$.
\end{lemma}
\begin{proof}
 Choose an irreducible curve $\widetilde{C}$ in $\delta_X^{-1}(C)\subset V$ such that $\delta_X(\widetilde{C}) = C$. 
 Then there exists a positive number $m$ such that $\delta_{X*}([\widetilde{C}]) = m\cdot [C]$.
 
 By the projection formula and Lemma \ref{lemma_pb_of_theta} we get the following
 \begin{equation*}
  (\delta_Y^*(\theta_1(\psi)) + E)\cdot [\widetilde{C}] = \delta_X^*(\theta_1(\varphi))\cdot[\widetilde{C}] = m\cdot\theta_1(\varphi)\cdot[C] <0.
 \end{equation*}
 Since $\theta_1(\psi)$ is a nef class this implies that $E\cdot [\widetilde{C}]<0$. Since $E$ is an effective exceptional divisor of $\delta_X$
 this is possible only if $\widetilde{C}$ is included in the support of $E$. As $E$ is contracted by $\delta_X$, we get 
 \begin{equation*}
  C = \delta_X(\widetilde{C}) \subset \delta_X(E) \subset \Ind(\alpha).
 \end{equation*}
 Thus, we get that a $\theta_1(\varphi)$-negative curve lies in the indeterminacy locus of any regularization map.
\end{proof}

\begin{lemma}\label{lemma_cond_implies}
 Assume that $X$ and $Y$ are smooth varieties, $\dim(X) = 3$, $\varphi\colon X\dashrightarrow X$ is a pseudo-auto\-mor\-phism, 
 $\alpha\colon X\dashrightarrow Y$ is a rational map and $\psi\colon Y\dashrightarrow Y$ is a birational automorphism
  such that $\alpha\circ \varphi = \psi\circ\alpha$. Let either $\psi$ be a  regular automorphism or $Y$ be a surface.
 If $D$ is a nef divisor class on $Y$ and $C$ is a curve on $X$ such that $\alpha^*D\cdot[C]<0$ and $C\subset\bigcap_{i>N}\Ind(\varphi^i)$, 
 then for any $m>0$ except a finite set we have~\mbox{$C\subset\Ind(\varphi^{-m})$.}
\end{lemma}
\begin{proof}
  Assume that $C$ is an irreducible $\alpha^*D$-negative curve and
 the following set is infinite:
 \begin{equation*}
  I = \{ m\in \ZZ | \ C\not\subset \Ind(\varphi^{m}) \}.
 \end{equation*} 
 By assumption we have that $I$ is included in $\{ m\leqslant N\}$ for some integer $N>0$.
 
 For each $-m\in I$ we have that $C$ does not lie in $\Ind(\varphi^{-m})$. Denote by $C_{-m}$ the proper image of the curve~$C$ under $\varphi^{-m}$ for each $-m\in I$.
 Since $\varphi$ is a pseudo-automorphism the curve $C_{-m}$ does not lie in~$\Ind(\varphi^{m})$ and its proper image under $\varphi^m$ is $C$.
 
 By Lemma \ref{lemma_theta-negative_curves} the curve $C$ lies in $\Ind(\alpha)$. 
 Since $\Ind(\psi^n)$ does not contain curves for any $n$ this implies
 that
 \begin{equation*}
  C\subset \Ind(\psi^{-m}\circ\alpha) = \Ind(\alpha\circ\varphi^{-m}),
 \end{equation*}
 for all $-m\in I$. Thus, the curve $C_{-m}$ lies in $\Ind(\alpha)$ for all $-m$ in $I$.

 Since $\Ind(\alpha)$ contains only finite number of curves there is some $-m\in I$ such that $C_{-m} = C$.
 
 The proper image of $C_{-m}$ under $\varphi^m$ is $C$.
 Then $C = C_{-m}$ does not lie in $\Ind(\varphi^m)$ and also in $\Ind(\varphi^{k m})$ for all $k>0$. 
 However, this contradicts our assumption and concludes the proof.
\end{proof}

Now we are ready to prove the criterion for non-regularizable automorphisms.
\begin{proof}[Proof of Theorem \textup{\ref{thm_criterion}}]
%  The first assertion is true by Lemma \ref{lemma_theta-negative_in_ind}. 
 By Lemma \ref{lemma_theta-negative_in_ind} we only have to prove that if there is a curve $C$ on $X$ such that  Condition \ref{condition} is satisfied,
 then the pseudo-automorphism $\varphi$  can not be regularized. 
 
 By contraction we assume that there exists a regularization of $\varphi$ as on the 
 diagram \eqref{diag_elimination}. By Lemma \ref{lemma_pb_of_theta} we get that $\theta_1(\varphi) = \alpha^*\theta_1(\psi)$. 
 Moreover, by \cite{Diller_Favre} the class $\theta_1(\psi)$ is nef. Then Lemma \ref{lemma_cond_implies}  applied to the 
 divisor $\theta_1(\psi)$ and the curve $C$ from Condition \ref{condition} leads us to the contradiction.
%  Assume that $C$ is an irreducible $\theta_1(\varphi)$-negative curve and
%  the following set is infinite:
%  \begin{equation*}
%   I = \{ m\in \ZZ | \ C\not\subset \Ind(\varphi^{m}) \}.
%  \end{equation*} 
%  By Lemma \ref{lemma_theta-negative_in_ind} we have that $I$ is included in $\{ m\leqslant N\}$ for some integer $N>0$.
%  
%  For each $-m\in I$ we have that $C$ does not lie in $\Ind(\varphi^{-m})$. Denote by $C_{-m}$ the proper image of the curve~$C$ under $\varphi^{-m}$ for each $-m\in I$.
%  Since $\varphi$ is a pseudo-automorphism the curve $C_{-m}$ does not lie in~$\Ind(\varphi^{m})$ and its proper image under $\varphi^m$ is $C$.
%  
%  By Lemma \ref{lemma_theta-negative_curves} the curve $C$ lies in $\Ind(\alpha)$. 
%  Since $\psi$ is regular this implies
%  that
%  \begin{equation*}
%   C\subset \Ind(\psi^{-m}\circ\alpha) = \Ind(\alpha\circ\varphi^{-m}),
%  \end{equation*}
%  for all $-m\in I$. Thus, the curve $C_{-m}$ lies in $\Ind(\alpha)$ for all $-m$ in $I$.
%  
%  
%  Since $\Ind(\alpha)$ contains only finite number of curves, Condition \ref{condition} holds only if there is some $-m\in I$ such that $C_{-m} = C$.
%  
%  The proper image of $C_{-m}$ under $\varphi^m$ is $C$.
%  Then $C = C_{-m}$ does not lie in $\Ind(\varphi^m)$ and also in $\Ind(\varphi^{k m})$ for all $k>0$. However, this contradicts 
%  Lemma \ref{lemma_theta-negative_in_ind}. Thus, there is no regularization of $\varphi$.
\end{proof}

\subsection{Proof of Theorem \ref{thm_primitivity}}
 We start with considering the situation when $f$ is algebraically stable i.e. when  we have an equality $(f^*)^n = (f^n)^*$ of endomorphisms of $N^1(S)$ 
 for all integers $n\in \ZZ$. In fact, all birational automorphisms of surfaces are conjugated to algebraically stable ones by \cite[Theorem~0.1]{Diller_Favre}.
 
 If the pseudo-automorphism $\varphi$ preserves the fibration $\pi$, then we can construct a good model of $S$.
 \begin{lemma}\label{lemma_phi_pseudo_f_AS}
 Assume that $\varphi\colon X \dashrightarrow X$ is a pseudo-automorphism of a smooth threefold, $\pi\colon X \dashrightarrow S$ is a dominant
 rational map to a smooth surface $S$ and $f\colon S \dashrightarrow S$ is a birational  automorphism of $S$ such that $\pi\circ \varphi = f\circ \pi$. 
 Then there exists a birational morphism $\delta\colon \widetilde{S}\to S$ such that the automorphism $\widetilde{f} = \delta^{-1}\circ f\circ \delta$ of $\widetilde{S}$ 
 is algebraically stable and if we denote $\widetilde{\pi} = \delta^{-1}\circ\pi$, then
 \begin{equation*}
  \left(\widetilde{f}\circ\widetilde{\pi}\right)^* = \widetilde{\pi}^*\circ \widetilde{f}^*.
 \end{equation*}
 \end{lemma}
\begin{proof}
 Denote by $\Gamma$ the resolution of singularities of the graph of $\pi\colon X\dashrightarrow S$ and by $p\colon \Gamma\to X$ and $q\colon \Gamma\to S$ 
 the projections to $X$ and $S$ respectively:
 \begin{equation*}
  \xymatrix{
  &\Gamma\ar[ld]_{p}\ar[rd]^{q} & \\X\ar@{-->}[rr]^{\pi} &&S
  }
 \end{equation*}
 Consider the set of points $Z$ in $S$ such that for any $z\in Z$ the fiber $q^{-1}(z)$ is a divisor in $\Gamma$. 
 Since $\Gamma$ is irreducible $Z$ is finite. Denote by $\delta'\colon S'\to S$ the blow-up of $S$ in $Z$. Then by \cite[Theorem 0.1]{Diller_Favre} there 
 exists a birational morphism $\delta''\colon \widetilde{S}\to S'$ such that the automorphism $\widetilde{f}\colon \widetilde{S}\dashrightarrow \widetilde{S}$
 induced by $f$ is algebraically stable.
 
 Denote by $\delta$ the composition $\delta'\circ \delta''$ and by $\widetilde{\pi}$ the composition $\delta^{-1}\circ\pi$ and consider the following diagram:
 \begin{equation*}
  \xymatrix{
  X\ar@{-->}[r]^{\varphi} \ar@{-->}[d]_{\widetilde{\pi}} &X \ar@{-->}[d]^{\widetilde{\pi}}\\ \widetilde{S} \ar@{-->}[r]^{\widetilde{f}} &\widetilde{S}
  }
 \end{equation*}
  By construction the proper image of any divisor on $X$ is not a point on $\widetilde{S}$. In particular, the proper image of any divisor on $X$
  does not lie in $\Ind(\widetilde{f})$. Thus, by Lemma \ref{lemma_composition_of_pb} we get that $\widetilde{\pi}^*\circ \widetilde{f}^* = (\widetilde{f}\circ\widetilde{\pi})^*$.
\end{proof}

% \begin{lemma}\label{lemma_phi_pseudo_f_AS}
%  Assume that $\varphi\colon X \dashrightarrow X$ is a pseudo-automorphism of a smooth threefold, $\pi\colon X \dashrightarrow S$ is a dominant
%  rational map to a smooth surface $S$ and $f\colon S \dashrightarrow S$ is a algebraically stable birational  automorphism of $S$. 
%  If $\pi\circ \varphi = f\circ \pi$, then $f$ is regular.
% \end{lemma}
% \begin{proof}
%  Consider a point $p$ in $\Ind(f)$. Then there exists a curve $C$ in $S$ such that proper image of $C$ under~$f^{-1}$ equals $p$.
%  Take a divisor $D$ in $\pi^{-1}(C)$ such that the proper image of $D$ equals $C$. 
%  
%  Denote by $E$ the proper image of $D$ under $\varphi^{-1}$. Since $\varphi$ is a pseudo-automorphism we get that $E$ is divisor; the
%  proper image of $E$ under $\pi$ equals the point $p$.
% 
%  Since $f$ is an algebraically stable automorphism by \cite[Theorem 1.14]{Diller_Favre} the point $p$ does not lie in $\Ind(f^{-n})$ for all numbers $n>0$.
%  Denote by $p_i$ the point $f^{-i}(p)$.
%  
%  Denote by $E_i$ the proper image of $E$ under the map $\varphi^{-i}$. They are divisors and proper transform of $E_i$ under $\pi$ equals $p_i$ for each $i<0$; 
%  in particular, $E_i$ is exceptional divisor for $\pi$. Thus, for some $N>0$ we get that $E_N = E$.
%  
%  Therefore, we get two divisors $D$ and $E_{N-1}$, they are distinct since dimensions of their proper images under $\pi$ are distinct.
%  Their proper images under $\varphi^{-1}$ coinside and equal $E$. This contradicts to the fact that~$\varphi$ is a pseudo-automorphism.
% \end{proof}

This lemma allows us to prove Theorem \textup{\ref{thm_primitivity}}.
\begin{proof}[Proof of Theorem \textup{\ref{thm_primitivity}}]
  Assume that there exists a surface $S$,  a rational map $\pi\colon X \dashrightarrow S$ and a birational automorphism~\mbox{$f\colon S\dashrightarrow S$} 
 such that~\mbox{$\pi\circ \varphi = f\circ\pi$}. By \cite[Theorem 1.1]{Truong_rel_deg} 
 since the relative dimension of~$\pi$ equals 1 we have the following equality:
 \begin{equation*}
 1<\lambda_1(\varphi) = \lambda_1(f). 
 \end{equation*}

 By~Lemma \ref{lemma_phi_pseudo_f_AS} we can replace $S$ by its birational model such that $\pi^*\circ f^* = (f\circ\pi)^*$ and $f$ is algebraically stable.
%  Replacing $S$ by this birational model we get the 
%  following commutative diagram:
%  \begin{equation*}
%   \xymatrix{
%   X \ar@{-->}[d]^{\pi} \ar@{-->}[r]^{\varphi} & X \ar@{-->}[d]^{\pi} \\ S \ar[r]^{f} & S
%   }
%  \end{equation*}
 Since $\varphi$ is pseudo-automorphism no divisor in $X$ maps under $\varphi$ to a component of $\Ind(\pi)$.
 Then by Lemmas \ref{lemma_composition_of_pb} and \ref{lemma_phi_pseudo_f_AS} we get the following.
 \begin{equation}\label{eq_pullback_for_nonprimitive}
 \begin{split}
  \varphi^*(\pi^*(\theta_1(f))) = (\pi\circ\varphi)^*(\theta_1(f)) &= (f\circ\pi)^*(\theta_1(f)) =\\&= \pi^*(f^*(\theta_1(f))) 
  = \lambda_1(f)\cdot\pi^*(\theta_1(f))= \lambda_1(\varphi)\cdot\pi^*(\theta_1(f)).
 \end{split}
 \end{equation}
 Then Theorem \ref{thm_Truong_condition} implies that the class~\mbox{$\theta_1(\varphi)$} is proportional to $\pi^*(\theta_1(f))$
 with strictly positive coefficient.
 Since the divisor $\theta_1(f)$ is a nef divisor we get the contradiction 
 with Condition \ref{condition} by Lemma \ref{lemma_cond_implies}. 
\end{proof}

 \section{Blanc's pseudo-automorphism} 
\label{sect_Blanc}

\subsection{Construction} \label{subsect_blank_example} This family of pseudo-automorphisms is described in the paper of Blanc~\cite{Blanc_Pseudo}.
We recall the construction of the pseudo-automorphism only in dimension $3$, though, in other dimensions everything is similar.
We consider a smooth cubic hypersurface $Q$ in $\p^3$. To each smooth point $p\in Q$ we associate a birational automorphism of the projective
space
\begin{equation*}
 \sigma_p\colon \p^3 \dashrightarrow \p^3.
\end{equation*}
If $L$ is a general line passing through $p$ and intersecting $Q$ in three
distinct points $p$, $q_1$, $q_2$, then $\sigma_p|_L$ is a unique involution of $L\cong\p^1$ such that $\sigma_p|_L(q_i) = q_i$ for $i = 1,2$.
Thus, $\sigma_p$ is a birational involution of~$\p^3$, it preserves pointwise an open subset of $Q$ and its indeterminacy locus consists 
of the point $p$ and an irreducible curve~\mbox{$\Gamma\subset Q$} of degree 6, see \cite[Section 2]{Blanc_Pseudo}.

Consider now $k$  distinct points $p_1,\dots, p_k$ on $Q$; denote by $\Gamma_1, \dots, \Gamma_k$ the curves
in the base loci of the associated involutions $\sigma_{1},\dots,\sigma_{k}$. 
Denote by $\delta\colon X\to \p^3$ the successive blow-ups of points $p_1,\dots,p_k$,
then the proper transforms of curves $\Gamma_1,\Gamma_2, \dots, \Gamma_k$. Also denote by $\hat{\sigma}_{i}$ the birational
involution of $X$ induced by~$\sigma_i$:
\begin{equation*}
 \hat{\sigma}_{i}\colon X\dashrightarrow X.
\end{equation*}

We introduce the following condition on the points $p_1,\dots,p_k$ on a smooth cubic surface $Q$ in $\p^3$:
\begin{align}\label{eq_generality_of_Blanc_points}
 p_i\not\in \Gamma_j \ \ \forall i\ne j && |\Gamma_i\cap\Gamma_j|=6 \ \ \forall i\ne j  && \Gamma_i\cap\Gamma_j\cap\Gamma_k = \emptyset\ \ \forall \text{ distinct indices }i,j,k.
\end{align}
The condition $|\Gamma_i\cap\Gamma_j|=6$ means that $\Gamma_i$ and $\Gamma_j$ intersect transversally on the surface $Q$.
It is proved in~\cite{Blanc_Pseudo} that Condition \eqref{eq_generality_of_Blanc_points} is satisfied for a general set of points $p_1,\dots,p_k\in Q$. 
This gives us a construction of a pseudo-automorphism on a smooth rational threefold with dynamical degree greater than~1.
\begin{theorem}[{\cite[Theorem 1.2]{Blanc_Pseudo}}] \label{thm_blanc}
 Assume  Condition \eqref{eq_generality_of_Blanc_points} is satisfied.
 Then the composition of involutions $\hat{\sigma}_{1}\circ\dots\circ\hat{\sigma}_{k}$ 
 defines a pseudo-automorphism $\varphi\colon X\dashrightarrow X$. 
 If~\mbox{$k\geqslant 3$}, then we have  $\lambda_1(\varphi)=\lambda_2(\varphi)>1$.
\end{theorem}
 We can represent the pseudo-automorphism $\varphi$ as a composition of simpler birational maps.

 We denote by $\delta_i\colon X_i \to \p^3$ the successive blow-ups of points $p_1,\dots,p_k$,
 then the proper transforms of curves $\Gamma_i,\Gamma_1, \dots,\Gamma_{i-1},\Gamma_{i+1},\dots, \Gamma_k$.
 Note that $X_1 = X$ and $\delta_1 = \delta$.

 By \cite[Proposition 2.2]{Blanc_Pseudo} the involution $\sigma_{i}$ induces a regular automorphism $\tsi_i$ on the variety~$X_i$.
We denote by~\mbox{$\tau_{i,j}\colon X_i \to X_{j}$} the birational map induced by the identity map on $\p^3$.
Note that this map is a composition of flops described in Proposition \ref{prop_flop}.

\begin{lemma}
 In notations and assumptions of Theorem \textup{\ref{thm_blanc}} the pseudo-automorphism $\varphi$ is a composition~\eqref{eq_decomposition_of_varphi} of
 regular involutions $\tsi_i\colon X_i\to X_i$ and flops $\tau_{i-1,i}\colon X_{i-1}\dashrightarrow X_i$ and the following diagram commutes:
\begin{equation}\label{eq_decomposition_of_varphi}
 \xymatrix{X_1 \ar[r]^{\tsi_1}  \ar@/_1.5pc/@{-->}[rrrrrrr]_{\varphi}& X_1 \ar@{-->}[r]^{\tau_{1,2}}& X_2 \ar[r]^{\tsi_2}
 & \dots \ar[r]^{\tsi_{k-1}}& X_{k-1} \ar@{-->}[r]^{\tau_{k-1,k}}& X_k \ar[r]^{\tsi_k}&X_k \ar@{-->}[r]^{\tau_{k,1}} & X_1 }
\end{equation}
\end{lemma}

\subsection{Action of $\varphi^*$}
In this section we study the action of the inverse image of $\varphi$ on groups $N^1(X)$ and $N_1(X) = N^2(X)$.
First, let us define the generators of these groups.

We denote by $H$ the pullback of the hyperplane section from $\p^3$ to $X$, by $E_j$ the exceptional divisors over points $p_j$ and by $F_j$ the exceptional 
divisors over $\Gamma_j$. There exist classes $e_j,f_j\in N^2(X)$ which correspond to extremal rays of the contractions of $E_j$ and $F_j$ respectively. 
We denote by $h\in N^2(X)$ the image of the class of a line in~$\p^3$. 

\begin{lemma}\label{lemma_classes_blanc}
 Under the assumptions of Theorem \textup{\ref{thm_blanc}} the groups $N^1(X)$ and $N^2(X)$ are generated by the following classes:
 \begin{align*}
  N^1(X_i) &= \langle H, E_1,\dots, E_k, F_1,\dots, F_k\rangle;\\
  N^2(X_i) &= \langle h, e_1,\dots, e_k, f_1,\dots, f_k\rangle.
 \end{align*}
 Moreover, these two bases are ``almost dual'' in the sense that  
 \begin{align*}
  H\cdot h = 1; && & E_i\cdot h = 0;   && F_i\cdot h = 0 \text{ for all } i;\\
  H\cdot e_i=0; && &E_i\cdot e_i =-1; && F_i\cdot e_i = 0 \text{ for all } i;  \\
  && &E_i\cdot e_j = 0; &&F_i\cdot e_j = 0 \text{ for all } i\ne j;\\
  H\cdot f_i = 0;&& &E_i\cdot f_i = 0; &&F_i\cdot f_i = -1\text{ for all } i;\\
  && &E_i\cdot f_j = 0;&&F_i\cdot f_j = 0  \text{ for all } i\ne j.
 \end{align*} 
\end{lemma}

We also use another set of elements in the group $N^1(X)$ to describe the action of the inverse image $\varphi^*$. Define a class~\mbox{$\nu_j\in N^1(X)$} as follows:
\begin{equation*}
 \nu_j = 2H-2E_j - F_j.
\end{equation*}

\begin{lemma}[{\cite[Proposition 2.3]{Blanc_Pseudo}}]\label{lemma_action_of_sigma}
 Under the assumptions of Theorem \textup{\ref{thm_blanc}} the involution $\hat{\sigma}_{i}$ acts on~$N^1(X)$ as follows:
 \begin{align*}
  &\hat{\sigma}_{i}^*(D) = D+\nu_i, \text{ if }D = H \text{ or }E_i;\\
  &\hat{\sigma}_{i}^*(F_i) = F_i+2\nu_i;\\
  &\hat{\sigma}_{i}^*(D) = D, \text{ if }D = E_j \text{ or }F_j \text{ for } j\ne i;\\
  &\hat{\sigma}_{i}^*(\nu_i) = -\nu_i;\\
  &\hat{\sigma}_{i}^*(\nu_j) = \nu_j+2\nu_i.
 \end{align*}
\end{lemma}

Since we know the action of involutions on groups of classes of divisors of $X_i$ we are able to compute the first dynamical class of $\varphi$.

\begin{lemma}[{\cite[Proof of Proposition 2.3]{Blanc_Pseudo}}]\label{lemma_theta_1_varphi}
  Under the assumptions of Theorem \textup{\ref{thm_blanc}} the first dynamical class of $\varphi$ is as follows:
  \begin{equation}
   \theta_1(\varphi) = \sum_{i=1}^k \alpha_i \nu_i \in N^1(X),
  \end{equation}
 where $\alpha_1>\alpha_2>\dots>\alpha_k$ are strictly positive numbers and $ \alpha_1>\alpha_2+2\left(\sum_{i=3}^k\alpha_i\right)$.
%  Moreover, $\alpha_1\geqslant 2\sum_{i=2}^k \alpha_i - \alpha_2$.
\end{lemma}
This implies in particular that the class $\theta_1(\varphi)$ is not nef. By Condition \eqref{eq_generality_of_Blanc_points} 
curves $\Gamma_1$ and $\Gamma_2$ lie on $Q$ and their intersection is transversal in $Q$. Choose a point $q\in\Gamma_1\cap\Gamma_2$ and denote by $L$ the proper transform in $X$ 
of the line in $\p^3$ passing through $p_1$ and $q$.
\begin{corollary}\label{corollary_theta_not_nef}
 Under the assumptions of Theorem \textup{\ref{thm_blanc}} we have $\theta_1(\varphi)\cdot[L] < 0$.
\end{corollary}
\begin{proof}
 By formulas in \cite[Section 2]{Blanc_Pseudo} we see that $\tsi_1(L)$ is the fiber of $F_1$ over the point $q$ on $\Gamma_1$.
 Thus, by Proposition \ref{prop_flop} we get that $\sigma_1(L)$ is a curve of indeterminacy of $\tau_{12}$. Then by Proposition \ref{prop_flop} we see
 \begin{equation*}
  \tsi_{1*}[L] = f_1 - f_2.
 \end{equation*}
 Since $\tsi_1$ is an involution we get $\tsi_{1}^*[L] = \tsi_{1*}[L]$.
 Then by Lemma \ref{lemma_action_of_sigma} we have the following equalities:
 \begin{align*}
  &H\cdot [L]  = \tsi_1^*(H)\cdot \tsi_{1}^*[L] = (3H - 2E_1 - F_1)\cdot(f_1 - f_2) = 1;\\
  &E_1\cdot [L]= \tsi_1^*(E_1)\cdot \tsi_{1}^*[L] = (2H - E_1 - F_1)\cdot(f_1 - f_2) = 1;\\
  &F_1\cdot [L]= \tsi_1^*(F_1)\cdot \tsi_{1}^*[L] = (4H - 4E_1 - F_1)\cdot(f_1 - f_2) = 1;\\
  &F_2\cdot [L]= \tsi_1^*(F_2)\cdot \tsi_{1}^*[L] = F_2\cdot(f_1 - f_2) = 1.
  \end{align*}
 By Lemma \ref{lemma_classes_blanc} this implies that $[L] = h - e_1 -f_1 - f_2$. Then by Lemma \ref{lemma_theta_1_varphi} we get
 \begin{multline*}
  \theta_1(\varphi)\cdot [L] = \left( \sum_{i=1}^k \alpha_i \nu_i \right)\cdot (h - e_1 -f_1 - f_2) =\\
  = \sum_{i=1}^k \alpha_i (2H-2E_j - F_j)\cdot (h - e_1 -f_1 - f_2) = 2\sum_{i=1}^k\alpha_i -3\alpha_1 - \alpha_2 < 0.
 \end{multline*}
 Thus, $L$ is an effective curve which intersects the class $\theta_1(\varphi)$ negatively.
\end{proof}

\subsection{Composition of three involutions}

In these section we consider the Blanc's pseudo-automorphism for $k = 3$. We denote $F = \sigma_3\circ\sigma_2\circ\sigma_1$,
it is a birational automorphism of $\p^3$:
\begin{equation}\label{eq_F}
 \xymatrix{ X_1\ar[d]_{\Delta} \ar[r]^{\tsi_1} & X_1 \ar@{-->}[r]^{\tau_{12}} & X_2 \ar[r]^{\tsi_2} & X_2 \ar@{-->}[r]^{\tau_{23}} & 
 X_3 \ar[r]^{\tsi_3} & X_3 \ar@{-->}[r]^{\tau_{31}} & X_1 \ar[d]^{\Delta} \\ \p^3 \ar@{-->}[rrrrrr]^{F} &&&&&& \p^3}
\end{equation}
Choose homogeneous coordinates $x_0, x_1,x_2$ and $x_3$ on $\p^3$. Denote by $f$ the equation of the cubic $Q$:
\begin{equation*}
 f(x) = f(x_0,x_1,x_2,x_3) = \sum_{|I| = 3} a_I x^I.
\end{equation*}
We fix the centers of three involutions, these are points $p_1$, $p_2$ and $p_3$. Fix a point $q$ in the intersection of curves $\Gamma_1$ and $\Gamma_2$.
\begin{lemma}\label{lemma_coordinates}
 If the cubic surface $Q$ and points $p_1$, $p_2$ and $p_3$  are sufficiently general, 
 then Condition \eqref{eq_generality_of_Blanc_points} is satisfied and up to conjugating $\p^3$ by an element in $\PGL(4,\CC)$ we have 
 \begin{align*}
  q = (1:0:0:0); && p_1 = (0:1:0:0); && p_2 = (0:0:1:0); && p_3 = (0:0:0:1).
 \end{align*}
 In this situation we get the following conditions on coefficients of $f$:
 \begin{equation*}
  a_{3000} = a_{0300} = a_{0030} = a_{0003} = a_{2100} = a_{2010} = 0.
 \end{equation*} 
%  	Any smooth cubic surface with points $p_1$,$p_2$ and $p_3$ in a general position can be written in these coordinates.
\end{lemma}
\begin{proof}
 First, let us prove that curves $\Gamma_1$ and $\Gamma_2$ intersect transversally in all points for a general choice of the cubic surface $Q$ and points $p_1$ and $p_2$.
 
 The first generality condition is that $p_2$ does not lie on $\Gamma_1$.
 Then take any point $q$ in $\Gamma_1\cap\Gamma_2$. 
 Points $p_1$, $p_2$ and $q$ are not colinear since the line passing through $p_1$ and $q$ does not intersect $Q$ in any other points and $p_2\ne q$ by our condition.
 
 Then there exist an automorphism of the projective space $\p^3$ which maps points $q$, $p_1$ and $p_2$ to points $(1:0:0:0)$, $(0:1:0:0)$ and $(0:0:1:0)$.
 Since points $q$, $p_1$, $p_2$ lie on the cubic surface $Q$ this implies that in some coordinates we have $ a_{3000} = a_{0300} = a_{0030} = 0$. 
 
 Consider the affine chart $U_0 = \{x_0\ne 0\}$ of $\p^3$. This chart is isomorphic to $\mathbb{A}^3$ with coordinates $t_1,t_2,t_3$ where $t_i = \frac{x_i}{x_0}$ for all $i = 1,2,3$.
 By \cite{Blanc_Pseudo} the curve $\Gamma_i$ is the intersection of the surfaces $\{f=0\}$ and $\{\frac{\partial{f}}{\partial{x_i}}=0\}$ for $i = 1,2$. 
 Denote by $g(t_1,t_2,t_3)$ the equation induced by $f$ on the chart $U_0$.
 Then the tangent line $L_i$ to $\Gamma_i$ in the point $q = (0,0,0)\in U_0$ is as follows:
 \begin{equation*}
  T_{L_i, q} = \left\{ \sum_{j=1}^3 \frac{\partial{g}}{\partial{t_j}}(0,0,0) dt_j =  \sum_{l=1}^3 \frac{\partial^2{g}}{\partial{t_i}\partial{t_l}}(0,0,0) dt_l = 0\right\} 
  \subset T_{U_0,q} = \langle dt_1, dt_2,dt_3\rangle.
 \end{equation*}
 The intersection of curves $\Gamma_1$ and $\Gamma_2$ in the point $q$ is not transversal if the tangent lines $T_{L_1,q}$ and $T_{L_2,q}$ coincide.
 Computing derivatives of $g$ we get that lines $T_{L_1, q}$ and $T_{L_2,q}$ coincide if the following polynomial vanishes:
 \begin{equation*}
  \det\begin{pmatrix}
              \frac{\partial{g}}{\partial{t_1}}(0,0,0) & \frac{\partial{g}}{\partial{t_2}}(0,0,0) & \frac{\partial{g}}{\partial{t_3}}(0,0,0)\\
              \frac{\partial^2{g}}{\partial{t_1}\partial{t_1}}(0,0,0) & \frac{\partial^2{g}}{\partial{t_1}\partial{t_2}}(0,0,0) & \frac{\partial^2{g}}{\partial{t_1}\partial{t_3}}(0,0,0)\\
              \frac{\partial^2{g}}{\partial{t_2}\partial{t_1}}(0,0,0) & \frac{\partial^2{g}}{\partial{t_2}\partial{t_2}}(0,0,0) &\frac{\partial^2{g}}{\partial{t_2}\partial{t_3}}(0,0,0)
             \end{pmatrix}
   =  a_{2001}(4a_{1200}a_{1020} -a_{1110}^2).
 \end{equation*}
 This implies that if $Q$ is sufficiently general and $p_2$ does not lie on $\Gamma_1$, then curves $\Gamma_1$ and $\Gamma_2$ intersect transversally.
 
 Now take three points $p_1, p_2$ and $p_3$ on a cubic $Q$ such that they are not collinear, $p_i$ does not lie on $\Gamma_j$ for $i\ne j$ and 
 curves $\Gamma_i$ and $\Gamma_j$ intersect transversally for all $1\leqslant i <j \leqslant 3$. Choose any point $q$ in $\Gamma_1\cap\Gamma_2$, then
 points $q,p_1,p_2$ and $p_3$ does not lie on one plane in $\p^3$. Thus, we can choose coordinates of $\p^3$ such that our four points  
 are $(1:0:0:0), (0:1:0:0), (0:0:1:0)$ and $(0:0:0:1)$ respectively.
 
 Since points $q,p_1,p_2$ and $p_3$ lie on $Q$ we get  $ a_{3000} = a_{0300} = a_{0030} = a_{0003} = 0$. Moreover, since the line passing through $p_i$ and $q$ 
 is tangent to $Q$ in $q$ for $i = 1$ and $2$ we get that $a_{2100} = a_{2010} = 0$.
\end{proof}

\begin{lemma}\cite[Section 2]{Blanc_Pseudo}\label{lemma_involutions_in_coordinates}
 Let $Q = \{f = 0\}$ be an equation of a cubic surface in $\p^3$ which contains points  $q = (1:0:0:0)$, $p_1 = (0:1:0:0)$, $p_2 = (0:0:1:0)$ and $p_3 = (0:0:0:1)$; 
 lines passing through $p_i$ and $q$ are tangent to $\Q_f$ in $q$ for $i = 1$ and $2$. Then the involutions associated with points $p_1$, $p_2$ and $p_3$
 are given by the following formulas:
 \begin{align*}
  \sigma_1(x_0,x_1,x_2,x_3) = \left(x_0\frac{\partial{f}}{\partial{x_1}}(x) : x_1\frac{\partial{f}}{\partial{x_1}}(x) - 2f(x) : 
  x_2\frac{\partial{f}}{\partial{x_1}}(x) : x_3\frac{\partial{f}}{\partial{x_1}}(x)\right);\\
  \sigma_2(x_0,x_1,x_2,x_3) = \left(x_0\frac{\partial{f}}{\partial{x_2}}(x) : x_1\frac{\partial{f}}{\partial{x_2}}(x) : 
  x_2\frac{\partial{f}}{\partial{x_2}}(x) - 2f(x) : x_3\frac{\partial{f}}{\partial{x_2}}(x)\right);\\
  \sigma_3(x_0,x_1,x_2,x_3) = \left(x_0\frac{\partial{f}}{\partial{x_3}}(x) : x_1\frac{\partial{f}}{\partial{x_3}}(x) : 
  x_2\frac{\partial{f}}{\partial{x_3}}(x) : x_3\frac{\partial{f}}{\partial{x_3}}(x) - 2f(x)\right).
 \end{align*}
\end{lemma}
\begin{proof}
 It suffices to prove that $\sigma_1$ is the necessary involution. Let us consider a line $l$ passing through the point $p_1 = (0:1:0:0)$.
 By Bezout theorem this line intersects the plane $\{ x_1 = 0\}$ in an only one point~\mbox{$(a:0:b:c)$}. Then the line $l$ is the following set of points:
 \begin{equation*}
  l = \{ (as:t:bs:cs)|\ (s:t)\in\p^1\}.
 \end{equation*}
 If we compute equations $f$ and $\frac{\partial{f}}{\partial{x_1}}$ in any point of the line $l$ we get the following:
 \begin{align*}
  f(as,t,bs,cs) &= \mu_1 s t^2  + \mu_2 s^2 t + \mu_3 s^3\\
  \frac{\partial{f}}{\partial{x_1}}(as,t,bs,cs) &= 2 \mu_1 st +\mu_2 s^2.
 \end{align*}
 Here $\mu_1, \mu_2$ and $\mu_3$ are some numbers in $\CC$ depending only on the line $l$ and the cubic $Q$.
 Then if we apply the formula of $\sigma_1$ to the point $(as:t:bs:cs)$, we get the following point on $\p^3$:
 \begin{multline*}
  \sigma_1(as,t,bs,cs) =\\= (as^2(2\mu_1t + \mu_2s): st(2\mu_1t + \mu_2s) - 2(\mu_1 s t^2  + \mu_2 s^2 t + \mu_3 s^3) : bs^2(2\mu_1t + \mu_2s):cs^2(2\mu_1t + \mu_2s)) =\\
  =(a(2\mu_1t + \mu_2s):-\mu_2 t - 2\mu_3 s: b(2\mu_1t + \mu_2s):c(2\mu_1t + \mu_2s)).
 \end{multline*}
 Then the map $\sigma_1$ induces an automorphism of the line $l$; in coordinates $s,t$ it can be written as the following matrix in $\PGL(2,\CC)$:
 \begin{equation*}
  \sigma_1|_{l} = \begin{pmatrix}
                   \mu_2 &-2\mu_3 \\ 2\mu_1 &-\mu_2
                  \end{pmatrix}
 \end{equation*}
 It is easy to check that this is an involution and its eigenvectors correspond to points $(s:t)$ such that $f(as,t,bs,cs) = 0$ and $(s:t)\ne (0:1)$.
\end{proof}

 Set $B = \{(2001), (1200), (1110), (1101), (0210)\}$. Denote by $R$ the following ring:
\begin{equation*}
 R = \ZZ[a_I]_{I\in B} = \ZZ[a_{2001}, a_{1200}, a_{1110}, a_{1101}, a_{0210}].
\end{equation*}
Consider a free commutative ring $R[X, Y]$. Set $w = a_{2001}a_{0210}+a_{1110}a_{1101}$ and denote by $\I$ the ideal $(4, 2w)$ in $R[X, Y]$, i.e., the ideal of polynomials whose coefficients lie in $(4, 2w) \subset R$.
%\begin{equation*}
% \I = (4, 2w)\subset R[X, Y].
%\end{equation*}
If $P$ is an element in $R[X,Y]$ we consider it as a polynomial of $X$ and $Y$ with coefficients in $R$. 

By $\deg_Y(P)$ we denote the degree of the polynomial $P$ in $Y$. By the leading term in $Y$ of the polynomial $P$
we call the polynomial $r(X)$ such that $\deg(P - r(X)Y^{\deg(P)}$ is strictly less than $\deg(P)$. 
We need the following useful property of degrees of polynomials:
\begin{lemma}\label{lemma_property_of_degree}
 Assume that $P_1$ and $P_2$ are two polynomials:
 \begin{equation*}
  P_i = r_{i}(X)Y^{d_i}  + Q_i(X, Y);\\
 \end{equation*}
 Assume also that the degree in $Y$ of all monomials in $Q_i$ is strictly less that $d_i$. If $d_1>d_2$ and $r_{1}(X)$ is a non-zero polynomial 
 which does not lie in $\I$, 
 then $P_1+P_2$ also does not lie in $\I$ and $\deg_Y(P_1+P_2) = d_1$.
\end{lemma}
\begin{proof}
 Since $X$ and $Y$ are free variables, then the polynomial $\sum_{i,j} r_{i,j}X^iY^j$ lies in $\I$ only if all coefficients $r_{i,j}$ lie in $\I$.
 Since $\deg(P_1)>\deg(P_2)$, then the leading term in $Y$ of $P_1+P_2$ coincides with the leading term of $P_1$. 
\end{proof}

Consider involutions $\sigma_{R1}, \sigma_{R2}$ and $\sigma_{R3}$ of $\p^3_{R[X,Y]}$ 
defined by formulas from Lemma \ref{lemma_involutions_in_coordinates} 
for a cubic surface $Q$ given by an equation $\sum_{I\in B} a_Ix^I = 0$. Denote by $F_R$ the composition of these involutions:
\begin{equation}\label{eq_F_R}
 F_R = \sigma_{R3}\circ\sigma_{R2}\circ\sigma_{R1} \colon \p^3_{R[X,Y]}\dashrightarrow \p^3_{R[X,Y]}
\end{equation}
Our goal is to compute the orbit of a point under the action of $F_R^{-1}$. We will consider points of the following form.
\begin{notation}\label{notation_point}
Let $p = ({M}_0+2N_0:{M}_1+2N_1:\widetilde{g}_2+2g_2:\widetilde{g}_3+ 2g_3)$ be a point in $\p_{R[X,Y]}^3$, 
such that ${M}_0, N_0, M_1, N_1, g_2,\widetilde{g}_2, g_3$ and $\widetilde{g}_3$ are elements of $R[X,Y]$ of the following form:
\begin{enumerate}
 \item[$(\mathrm{B}1)$]  Polynomials $\widetilde{g}_2$ and $\widetilde{g}_3$ are elements of~$\I$;
 also $g_2$ and $g_3$ lie in $2R[X, Y]$ and 
 the leading term in $Y$ of $\widetilde{g}_3$ do not lie in $4R[X, Y]$; 
 \item[$(\mathrm{B}2)$] The leading terms of polynomials ${M}_i(X, Y)$ 
 do not lie in the ideal $2R[X, Y]$ for $i = 0$ and $i = 1$.
 \item[$(\mathrm{B}3)$]  We have the following conditions on degrees of polynomials:
 \begin{align*}
  \deg_Y(\widetilde{g}_2)\leqslant \deg_Y({M_0}) < \deg_Y(\widetilde{g}_3)< \deg_Y({M}_1).
 \end{align*} 
%  \item[(B4)] \label{it_P0_nonzero} The following polynomial does not lie in $2R[X,Y]$:
%  \begin{equation}\label{eq_M}
%  G
% %  (M_0, M_1, \widetilde{g}_2, \widetilde{g}_3) 
%  = 
%  a_{1200}w M_0 M_1^2 + \frac{1}{2}(a_{1110}\widetilde{g}_2 + a_{1101}\widetilde{g}_3)(a_{2001}a_{1110}M_0^2+ wM_0 M_1 
%  + a_{1101} a_{0210}M_1^2)\not\in 2R[X,Y].
%  \end{equation}
\end{enumerate}
\end{notation}
Now we are going to show that the image of the point in ~\mbox{$\p_{R[X,Y]}^3$} which satisfies conditions $(\mathrm{B}1 - \mathrm{B}3)$ 
under $F_R^{-1}$ still satisfies these conditions.
\begin{lemma}\label{lemma_preimage_of_point}
Assume that $p = ({M}_0+2N_0:{M}_1+2N_1:\widetilde{g}_2+2g_2:\widetilde{g}_3+ 2g_3)$ is a point in the projective space~\mbox{$\p_{R[X,Y]}^3$} 
such that polynomials  ${M}_0, N_0, M_1, N_1, g_2,\widetilde{g}_2, g_3$ and $\widetilde{g}_3$ satisfy conditions $(\mathrm{B}1 - \mathrm{B}3)$.
Then the point $p$ does not lie in $\Ind(F_R^{-1})$.  
Moreover, $F_{R}^{-1}(p) =   ({M}'_0+2N'_0:{M'}_1+2N'_1:\widetilde{g}'_2+2g'_2:\widetilde{g}'_3+ 2g'_3)$ 
where polynomials ${M}'_0, N'_0, M'_1, N'_1, g'_2, \widetilde{g}'_2, {g}'_3$ and $\widetilde{g}'_3$ satisfy conditions $(\mathrm{B}1 - \mathrm{B}3)$.
\end{lemma}
\begin{proof}
 In view of condition $(\mathrm{B}3)$ one has $ \deg_Y(M_0) = d_0$, $ \deg_Y(M_1) = d_1$, $\deg_Y(\widetilde{g}_2) = d_2$, 
 $\deg_Y(\widetilde{g}_3) = d_3$ and we have the following inequality:
 \begin{align}\label{eq_degrees}
  d_2\leqslant d_0 <d_3<d_1.
 \end{align}
 In order to compute~$F_{R}^{-1}(p)$ we use the formulas from Lemma~\ref{lemma_coordinates} 
 associated with the cubic surface given by the equation $\sum_{i\in B} a_I x^I$. 
 Then $F_{R}^{-1}(p) = (q_0:q_1:q_2:q_3)$, where the expressions $q_i$ can be computed in Sage explicitly. 
 
 First thing we note is that all polynomials $q_0$, $q_1$, $q_2$ and $q_3$ lie in $2R[X,Y]$. Our goal is to 
 construct polynomials ${M}'_0, N'_0, M'_1, N'_1, g'_2,\widetilde{g}'_2, g'_3$ and $\widetilde{g}'_3$
 such that 
  \begin{equation*}\label{eq_result_of_computation}
  \begin{split}
   & q_0 = 2(M'_0 + 2N'_0);\\
   & q_1 = 2 (M'_1 + 2N'_1);\\
   & q_2 = 2 (\widetilde{g}'_2 + 2g'_2);\\
   & q_3 = 2(\widetilde{g}'_3 + 2g'_3).
  \end{split}
 \end{equation*}
 We are going to do it considering all components $q_0$, $q_1$, $q_2$ and $q_3$ one by one.
 
 \medskip
 
 Consider the component $q_0$. By formulas in Lemma \ref{lemma_involutions_in_coordinates} we can see that 
 this is a homogeneous polynomial  of ${M}_0, N_0, M_1, N_1, g_2,\widetilde{g}_2, g_3$ and $\widetilde{g}_3$ of degree $27$.
 We define $M'_0$ to be the sum of all monomials in $\frac{1}{2}q_0$  except those divisible 
 by $2, g_2,\widetilde{g}_2, g_3, \widetilde{g}_3$ or $\frac{1}{2}m$ where $m$ is a degree 2 monomial 
 of $g_2,\widetilde{g}_2, g_3$ and $\widetilde{g}_3$.
 Note that by construction  $\frac{1}{2}q_0 - M'_0$ lies in $2R[X, Y]$ since so do $g_2,\widetilde{g}_2, g_3$ 
 and $\widetilde{g}_3$. Denote $N'_0 = \frac{1}{2}(\frac{1}{2}q_0 - M'_0)$.
 The computation in Sage shows:
 \begin{equation}\label{eq_M0_new}
   M'_0 = a_{1101}^{8} a_{0210}^{2} M_{0}^{11} M_{1}^{15}(w a_{1200}  M_{0}
+ \frac{1}{2} a_{1110} a_{1101} a_{0210} \widetilde{g}_{2} + \frac{1}{2} a_{1101}^{2} a_{0210} \widetilde{g}_{3}) + h_{M'_0},
 \end{equation}
 where $h_{M'_0}$ is a homogeneous polynomial of ${M}_0, N_0, M_1, N_1, g_2,\widetilde{g}_2, g_3$ and $\widetilde{g}_3$ of degree $27$
 and the degree in $M_1$ of all monomials is strictly less than $15$. By Lemma \ref{lemma_property_of_degree} and \eqref{eq_degrees} 
 we get that the leading term in $Y$ of $M'_0$ does not lie in $2R[X, Y]$ and
 \begin{equation}\label{eq_deg_M0}
  \deg_Y(M'_0) = 11 d_0 + 15 d_1 + d_3.
 \end{equation}
 
 Consider the component $q_1$. This is a homogeneous polynomial 
 of ${M}_0, N_0, M_1, N_1, g_2,\widetilde{g}_2, g_3$ and $\widetilde{g}_3$ of degree $27$.
 We define $M'_1$ to be the sum of all monomials in $\frac{1}{2}q_1$  except those divisible 
 by $2, g_2,\widetilde{g}_2, g_3, \widetilde{g}_3$ or $\frac{1}{2}m$ where $m$ is a degree 2 monomial 
 of $g_2,\widetilde{g}_2, g_3$ and $\widetilde{g}_3$.
 Note that by construction  $\frac{1}{2}q_1 - M'_1$ lies in $2R[X, Y]$. Denote $N'_1 = \frac{1}{2}(\frac{1}{2}q_1 - M'_1)$.
 The computation in Sage shows:
 \begin{equation}\label{eq_M1_new}
  M'_1 = a_{1200} a_{1101}^{9} a_{0210}^{3} M_{0}^{10} M_{1}^{17} + h_{M'_1},
 \end{equation}
 where $h_{M'_1}$ is a homogeneous polynomial of ${M}_0, N_0, M_1, N_1, g_2,\widetilde{g}_2, g_3$ and $\widetilde{g}_3$ of degree $27$
 and the degree in $M_1$ of all monomials is strictly less than 17. By Lemma \ref{lemma_property_of_degree} and \eqref{eq_degrees} 
 we get that the leading term in $Y$ of  $M'_1$ does not lie in $2R[X, Y]$ and
 \begin{equation}\label{eq_deg_M1}
  \deg_Y(M'_1) = 10 d_0 + 17 d_1.
 \end{equation}

 Consider the component $q_2$. This is a homogeneous polynomial 
 of ${M}_0, N_0, M_1, N_1, p_2 = \widetilde{g}_2+2g_2$ and $p_3 = \widetilde{g}_3 + 2g_3$ of degree $27$. 
 We define $q'_2$ to be the sum of all monomials of $q_2$ except those divisible 
 by $8, 2p_2, 2p_3$ or a degree 2 monomial of $p_2$ and $p_3$.
 Note that by construction $q_2 - q'_2$ lies in $2\I$.
 The computation in Sage shows:
 \begin{equation}\label{eq_q2_in_I}
  q'_2 = 4w a_{1200}^2 M_0^{12} {M}_1^6   (a_{1110} M_0 + a_{0210}{M}_1) (a_{2001}^8  M_0^8 + a_{1101}^8 {M}_1^8)
 \end{equation}
 Since $q'_2$ is divisible by $4w$ it lies in $2\I$. Therefore, by construction $q_2$ lies in $2\I$.
 Define $\widetilde{g}'_2$ to be the sum of all monomials of $\frac{1}{2}q_2$ except those divisible 
 by $4, 2g_2,2\widetilde{g}_2, 2g_3, 2\widetilde{g}_3$ or a degree  2 monomial of $g_2,\widetilde{g}_2, g_3$ and $\widetilde{g}_3$.
 Note that by construction  $g_2 = \frac{1}{2}(\frac{1}{2}q_2 - \widetilde{g}'_2)$ lies in $2R[X, Y]$.
 The computation in Sage shows:
 \begin{multline}\label{eq_g2_new}
  \widetilde{g}'_2 = a_{0210} a_{1101}^8 M_0^{10}  M_1^{15} (2w a_{1200}^2  M_0^2  +
  a_{2001}a_{1200} a_{0210}^2 M_0 \widetilde{g}_2 +\\+ \frac{1}{2}a_{1110} a_{1101}a_{0210}^2 \widetilde{g}_2^2 +
  a_{1200}a_{1101}^2 a_{0210} M_0 \widetilde{g}_3 + \frac{1}{2}a_{1101}^2 a_{0210}^2 \widetilde{g}_2 \widetilde{g}_3))+ h_{\widetilde{g}'_2},
 \end{multline}
 where $h_{\widetilde{g}'_2}$ is a homogeneous polynomial of ${M}_0, N_0, M_1, N_1, g_2,\widetilde{g}_2, g_3$ and $\widetilde{g}_3$ of degree $27$
 and the degree in $M_1$ of all monomials is strictly less than $15$. By  \eqref{eq_degrees} and
 we get that
 \begin{equation}\label{eq_deg_g2}
  \deg_Y(\widetilde{g}'_2) \leqslant 11 d_0 + 15 d_1+ d_3.
 \end{equation}

 Consider the component $q_3$.This is a homogeneous polynomial 
 of ${M}_0, N_0, M_1, N_1, p_2 = \widetilde{g}_2+2g_2$ and $p_3 = \widetilde{g}_3 + 2g_3$ of degree $27$. 
 We define $q'_3$ to be the sum of all monomials of $q_3$ except those divisible 
 by $8, 2p_2, 2p_3$ or a degree 2 monomial of $p_2$ and $p_3$.
 Note that by construction $q_3 - q'_3$ lies in $2\I$.
 The computation in Sage shows:
 \begin{equation}\label{eq_q3_in_I}
 \begin{split}
  q'_3 = 4w a_{1200}^2   M_0^{11} {M}_1^7 (a_{2001} M_0 + a_{1101}{M}_1) &
   (a_{1110}^2 M_0^2 +\\ + a_{0210}^2{M}_1^2) 
   &(a_{2001}^2  M_0^2 + a_{1101}^2 {M}_1^2) (a_{2001}^4 M_0^4 + a_{1101}^4{M}_1^4).
  \end{split}
 \end{equation}
 Since $q'_3$ is divisible by $4w$ it lies in $2\I$. Therefore, by construction $q_3$ lies in $2\I$.
 
 Define $\widetilde{g}'_3$ to be the sum of all monomials of $\frac{1}{2}q_3$ except those divisible 
 by $4, 2g_2,2\widetilde{g}_2, 2g_3, 2\widetilde{g}_3$ or degree 2 monomial of $g_2,\widetilde{g}_2, g_3$ and $\widetilde{g}_3$.
 Note that by construction $g_3 = \frac{1}{2}(\frac{1}{2}q_3 - \widetilde{g}'_3)$ lies in $2R[X, Y]$. Moreover,
 the Sage computation shows:
 \begin{equation}\label{eq_g3_new}
 \begin{split}
  \widetilde{g}'_3 = a_{1200} a_{0210}^2  a_{1101}^7 M_0^{10} M_1^{16} 
  (2w a_{1200} M_{0} + a_{1110} a_{1101} a_{0210} \widetilde{g}_{2} + a_{1101}^{2} a_{0210} \widetilde{g}_{3}) + h_{\widetilde{g}'_3}.
 \end{split} 
 \end{equation}
 Here $h_{\widetilde{g}'_3}$ is a homogeneous polynomial of ${M}_0, N_0, M_1, N_1, g_2,\widetilde{g}_2, g_3$ and $\widetilde{g}_3$ of degree $27$
 and the degree in $M_1$ of all monomials is strictly less than $16$. By Lemma \ref{lemma_property_of_degree} and \eqref{eq_degrees} 
 we get that the leading term in $Y$ of $ \widetilde{g}'_3$ does not lie in $4R[X, Y]$ and
 \begin{equation}\label{eq_deg_g3}
  \deg_Y(\widetilde{g}'_3) = 10 d_0 + 16 d_1+ d_3.
 \end{equation}

 We have constructed polynomials ${M}'_0, N'_0, M'_1, N'_1, g'_2,\widetilde{g}'_2, g'_3$ and $\widetilde{g}'_3$,
 it remains to show conditions $(\mathrm{B}1 - \mathrm{B}3)$ for these polynomials.

 Polynomials $\widetilde{g}'_2+ 2g'_2$ and $\widetilde{g}'_3+ 2g'_3$ lie in $\I$ by equations \eqref{eq_q2_in_I} and \eqref{eq_q3_in_I}.
 Polynomials $g'_2$ and $g'_3$ lie in $2R[X, Y]$ by the construction.
 Also by \eqref{eq_g2_new} and \eqref{eq_g3_new} we get that the leading terms of $\widetilde{g}_2$ and $\widetilde{g}_3$ 
 are not divisible by $4$. Thus, the condition $(\mathrm{B}1)$ is satisfied.

 By \eqref{eq_M0_new} and \eqref{eq_M1_new} we get that the leading monomials of $M'_0$ and $M'_1$
 do not lie in $2R[X, Y]$. Thus, the condition $(\mathrm{B}2)$ is satisfied.

 Finally the condition $(\mathrm{B}3)$ follows from the degree computation \eqref{eq_deg_M0}, \eqref{eq_deg_M1},  \eqref{eq_deg_g2}, \eqref{eq_deg_g3}
 and the assumption \eqref{eq_degrees}.

 Since by our computation polynomials $q_0, q_1, q_2$ and $q_3$ are non-zero, we get that the point $p$ does not lie in $\Ind(F_R^{-1})$.
\end{proof}
\begin{corollary}\label{corollary_pointXY00}
 Let $p = (X:Y:0:0)$ be a point in $\p^3_{R[X,Y]}$ and $F_R$ be a birational automorphism as in \textup{\eqref{eq_F_R}}. 
 Then $p$ does not lie in $\Ind(F_R^{-n})$ for all $n>0$.
\end{corollary}
\begin{proof}
 Consider $F_R^{-1}(p)= q = (q_0:q_1:q_2:q_3)$. Below we construct polynomials $M'_0$, $N'_0$, $M'_1$, $N'_1$, $\widetilde{g}'_2$, $g'_2$,
 $\widetilde{g}'_3$ and $g'_3$  such that
   \begin{equation}\label{eq_preimage_of_XY00}
  \begin{split}
   & q_0 = 2(M'_0 + 2N'_0);\\
   & q_1 = 2 (M'_1 + 2N'_1);\\
   & q_2 = 2 (\widetilde{g}'_2 + 2g'_2);\\
   & q_3 = 2(\widetilde{g}'_3 + 2g'_3).
  \end{split}
 \end{equation}
 Consider the component $q_0$, it is a homogeneous polynomial in $X$ and $Y$ of degree $27$.
 Define $M'_0$ to be the sum of all monomials of $\frac{1}{2}q_0$ which coefficients are not divisible by $4$.
 Set $N'_0 = \frac{1}{2}(\frac{1}{2}q_0 - M'_0)$.
 Using the formula~\eqref{eq_M0_new} and substituting there $M_0 = X$, $M_1 = Y$, $\widetilde{g}_2 = \widetilde{g}_3 = 0$ we 
 can see that the leading term in $Y$ of $M'_0$ 
 is equal to $w a_{1200}a_{1101}^{8} a_{0210}^{2} X^{12} Y_{1}^{15}$. It does not lie in $2R[X, Y]$ and the degree in $Y$ of $M'_0$ is
 \begin{equation*}
  \deg_Y(M'_0) = 15.
 \end{equation*}
 Consider the component $q_1$, it is a homogeneous polynomial in $X$ and $Y$ of degree $27$.
 Define $M'_1$ to be the sum of all monomials of $\frac{1}{2}q_1$ which coefficients are not divisible by $4$.
 Set $N'_1 = \frac{1}{2}(\frac{1}{2}q_1 - M'_1)$.
 Using the formula~\eqref{eq_M1_new} and substituting there $M_0 = X$, $M_1 = Y$, $\widetilde{g}_2 = \widetilde{g}_3 = 0$ we 
 can see that the leading term in $Y$ of $M'_1$ 
 is equal to~$a_{1200} a_{1101}^{9} a_{0210}^{3} X^{10} Y^{17}$. It does not lie in $2R[X, Y]$ and the degree in $Y$ of $M'_1$ is
 \begin{equation*}
  \deg_Y(M'_1) = 17.
 \end{equation*}
 Consider the component $q_2$. By the formula \eqref{eq_q2_in_I} we see that $q_2$ lies in the ideal $\I$.
 Define $\widetilde{g}'_2$ to be the sum of all monomials of $\frac{1}{2}q_2$ which are not divisible by $4$.
 Set $g'_2 = \frac{1}{2}(\frac{1}{2}q'_2 - \widetilde{g}'_2)$.  Using the formula~\eqref{eq_g2_new} and 
 substituting there $M_0 = X$, $M_1 = Y$, $\widetilde{g}_2 = \widetilde{g}_3 = 0$ we 
 can see that the leading term in $Y$ of $\widetilde{g}'_2$ 
 is equal to~$2w a_{1200}^2 a_{0210} a_{1101}^8 X^{12}  Y^{15}$. Then the degree in $Y$ of $\widetilde{g}'_2$ is
 \begin{equation*}
  \deg_Y(\widetilde{g}'_2) = 15.
 \end{equation*}
 Consider the component $q_3$. By the formula \eqref{eq_q3_in_I} we see that $q_3$ lies in the ideal $\I$.
 Define $\widetilde{g}'_3$ to be the sum of all monomials of $\frac{1}{2}q_3$ which are not divisible by $4$.
 Set $g'_3 = \frac{1}{2}(\frac{1}{2}q_3 - \widetilde{g}'_3)$.  
 Using the formula~\eqref{eq_g3_new} and 
 substituting there $M_0 = X$, $M_1 = Y$, $\widetilde{g}_2 = \widetilde{g}_3 = 0$ we 
 can see that the leading term in $Y$ of $\widetilde{g}'_3$ 
 is equal to~$2w a_{1200}^2 a_{0210}^2  a_{1101}^7 X^{11} Y^{16}$. It does not lie in $4R[X, Y]$ and  the degree in $Y$ of $\widetilde{g}'_3$ is
 \begin{equation*}
  \deg_Y(\widetilde{g}'_3) = 16.
 \end{equation*}
 Thus, one get that $q$ has the form \eqref{eq_preimage_of_XY00} and by the construction 
 polynomials  $M'_0$, $N'_0$, $M'_1$, $N'_1$, $\widetilde{g}'_2$, $g'_2$,
 $\widetilde{g}'_3$ and $g'_3$ satisfy conditions $(\mathrm{B}1 - \mathrm{B}3)$. 
 
 Thus, the point $p$ does not lie in the indeterminacy locus of $F_R^{-1}$. 
 Moreover, by Lemma \ref{lemma_preimage_of_point} the point $q = F_R^{-1}(p)$ does not lie in $\Ind(F_R^{-n})$ for all $n>0$.
 This finishes the proof.
\end{proof}

\begin{corollary}\label{corollary_line_on_vg_cubic}
 Let $Q$ be a very general cubic surface in $\p_{\CC}^3$ and let points $p_1$, $p_2$ and $p_3$ be general points on~$Q$. 
 Assume that $F$ is the birational automorphism of $\p_{\CC}^3$ as in \textup{\eqref{eq_F}}.
 Then the line $L\subset \p_{\CC}^3$ passing through $p_1$ and some point in $\Ind(\sigma_1)\cap\Ind(\sigma_2)$ does 
 not lie in $\Ind(F^{-n})$ for all $n>0$.
\end{corollary}
\begin{proof}
 Consider a $4$-dimensional complex vector space $V$ and associated projective space $\p(V) = \p^3$. A cubic surface in $\p(V)$ is given by an equation in $\p(S^3V^{\vee})$. 
 The coordinates of the space $\p(S^3V^{\vee})$ are coefficients $a_I$ of the cubic equation $\sum a_I x^I$. 
 Denote by $W$ the following subspace of $S^3V^{\vee}$:
 \begin{equation*}
  W = \left\langle a_I|\ I\neq (3000), (0300), (0030), (0003), (2100), (2010)\right\rangle\subset S^3V^{\vee}.
 \end{equation*}
 
 Consider a universal cubic surface $\Q\subset \p(V)\times \p(W)$. Denote by $\Pi\colon \Q\to \p(W)$ the projection.
 The fiber of $\Pi$ over a point $f \in \p(W)$ is a cubic surface $\Q_f = \{ f = 0\}$.
 Moreover, by construction $\Q_f$ contains points $q = (1:0:0:0)$, $p_1 = (0:1:0:0)$, $p_2 = (0:0:1:0)$ and $p_3 = (0:0:0:1)$; 
 lines passing through $p_i$ and $q$ are tangent to $\Q_f$ in $q$ for $i = 1$ and $2$. Note that for any point $f\in \p(W)$ 
 the point $q$ lies in $\Ind(\sigma_1)\cap\Ind(\sigma_2)$ where $\sigma_1$ and $\sigma_2$ are involutions associated with 
 the cubic equation $f$ and points $p_1$ and~$p_2$.
 
  Then we can define involutions $\sigma_{\Q 1}$, $\sigma_{\Q 2}$ and $\sigma_{\Q 3}$ of $ \p(V)\times \p(W)$ by formulas in Lemma \ref{lemma_involutions_in_coordinates}.
 Denote by $F_{\Q}$ the composition $\sigma_{\Q 1}\circ\sigma_{\Q 2}\circ\sigma_{\Q 3}$.
 
 Consider the following cubic equation $f_0$:
 \begin{equation*}
  f_0 = a_{2001}x_0^2x_3 +        a_{1200}x_0 x_1^2 + a_{1110}x_0x_1x_2 + a_{1101}x_0x_1x_3 + a_{0210}x_1^2x_2,
 \end{equation*}
 here coefficients $a_{2001}$, $a_{1200}$, $a_{1110}$, $a_{1101}$, $a_{0210}$ are non-zero complex numbers. 
 By Corollary \ref{corollary_pointXY00} we get that if $X$ and $Y$ are very general complex numbers then the point $(X:Y:0:0)$
 does not lie in the indeterminacy locus of $F^{-n}$ for all $n>0$. 
 
 This implies that the point $((X: Y:0:0), f_0)$ does not lie in $\Ind(F_{\Q}^{-n})$ for all $n>0$. 
 Denote by $L$ the line connecting $q$ and $p_1$. Then the point $(X:Y:0:0)$ lies on $L$ in $\p^3$. Therefore, 
 the subvariety~$L\times \p(W)$ does not lie in $\Ind(F_{\Q}^{-n})$ for all $n>0$. 
 
 Since $\bigcup_{n>0} \Ind(F_{\Q}^{-n})$ is a countable union of closed codimension $2$ subvarieties of $\p(V)\times \p(W)$, then for a very general point $f$
 in $\p(W)$ the line $L$ passing through $p_1$ and $q$ does not lie in $\Ind(F^{-n})$ for all $n>0$.
 
 Finally, since by Lemma \ref{lemma_coordinates} for a set of a cubic surface $Q$ and three general 
 points $p_1$, $p_2$ and $p_3$ there exists a choice of coordinates such that the equation of $Q$ 
 lies in $W$ and $p_1 = (0:1:0:0)$, $p_2 = (0:0:1:0)$ and $p_3 = (0:0:0:1)$ we get the result.
\end{proof}

 \subsection{Proof of Theorem \ref{theorem_blanc_example}}
 We consider the composition of three birational involutions $\sigma_1$, $\sigma_2$ and $\sigma_3$ on $\p^3$ 
 associated with points $p_1,p_2$ and $p_3$ on a smooth cubic surface $Q$.
 Without loss of generality we can assume that they are as in  Lemma \ref{lemma_coordinates}. Then involutions are defined
 by formulas in Lemma \ref{lemma_involutions_in_coordinates}. 
 
 If we prove that for some cubic surface and 3 points on it the composition $F = \sigma_3\circ\sigma_2\circ\sigma_1$ is not regularizable, then
 the same is true for any very general birational automorphism of this type. Thus, we assume that coefficients $a_I$ of the equation $f$ of 
 cubic $Q$ are transcendental and algebraically independent over $\mathbb{Q}$.
 
 We denote by $\delta\colon X\to\p^3$ the consequent blow-up of the proper preimages of curves $\Gamma_1,\Gamma_2$ and $\Gamma_3$.
 By Theorem \ref{thm_blanc} the birational automorphism $F$ induces a pseudo-automorphism $\varphi$ on $X$. By Theorems \ref{thm_blanc} and~\ref{thm_Truong_condition} 
 the class $\theta_1(\varphi)$ is well-defined and we can compute it by Lemma \ref{lemma_theta_1_varphi}.
 
 Consider the curve $L$ on $X$ which is the proper preimage under $\delta$ of the line $\delta(L)$ on $\p^3$ passing through points $p_1$ and $q$.
 By Corollary \ref{corollary_theta_not_nef} we get that $\theta_1(\varphi)\cdot [L]<0$.
 However, by Corollary \ref{corollary_line_on_vg_cubic} we get that  the curve $\delta(L)$ does not lie in $\Ind(F^{-N})$ for all $N>0$. 
 
 This implies that the curve $L$ does not lie in $\Ind(\varphi^{-N})$  
 for infinitely many numbers $N>0$. Thus, Condition \ref{condition} is satisfied for the pseudo-automorphism $\varphi$.
 Then by Theorem \ref{thm_criterion}  we get that there is no birational model of $X$ on which $\varphi$ defines a regular automorphism.
 Moreover, by Theorem \ref{thm_primitivity} this also implies that $\varphi$ does not preserve any fibration over a surface.

\section{Example of a regularizable pseudo-automorphism with non-nef class \texorpdfstring{$\theta_1$}{T\_1}} 
\label{sect_examples}
We recall here the construction from \cite{Oguiso-Truong} and \cite[Section 7]{Lesieutre_constraints} in order to give an example of a regularizable
(not regular) pseudo-automorphism $\varphi_+$ such that the first dynamical class $\theta_1(\varphi_+)$ is not nef.

First let us recall the construction of the regular primitive automorphism of a rational threefold with the dynamical degree greater than 1 described in \cite{Oguiso-Truong}.

In order to construct it we consider the lattice of  Eisenstein integers $\mathbb{Z}[j]$ in $\mathbb{C}$, here $j$ is a primitive cubic root of unity.
Denote by $E$ the elliptic curve $\CC/\ZZ[j]$.
The group  $G = \langle -j \rangle$ isomorphic to $\ZZ/6\ZZ$ acts on $E$ and this action induces a diagonal action of $G$ on the abelian 
variety $A = E\times E\times E$. 

We denote by $q\colon A\to A/G$ the quotient map. Since the action of $G$ was not free, $A/G$ is a singular variety; there are 27 
non-terminal isolated quotient singularities on~$A/G$. Denote by $\delta \colon X \to A/G$ the blow-up of all these points. 
Then $X$ is a smooth rational threefold  by \cite[Theorem 1.4]{Oguiso-Truong}.

Any matrix $M$ in $\SL(3,\ZZ)$ defines a regular automorphism of $A$. Since the action of $\SL(3,\ZZ)$ and $G$ commutes, this action extends 
to a regular action on $A/G$ and on $X$. We fix a matrix $M$ as in  \cite[Lemma 4.3]{Oguiso-Truong}. It is an integer invertible matrix 
such that all roots of its characteristic polynomial are distinct real numbers.
Then $M$ induces primitive regular automorphisms $\varphi_A$ and $\varphi$ of $A$ and $X$ respectively and 
\begin{equation*}
 \lambda_1(\varphi)^2 \geqslant \lambda_2(\varphi).
\end{equation*}

Denote by $C$ a proper image of a curve $E\times\{0\}\times\{0\}$ on $A$ under the finite rational map $\delta\circ q$.
Then~$C$ is a smooth rational curve on $X$.
By \cite[Section 7]{Lesieutre_constraints} there is a standard Atiyah flop in the curve $C$; i.e. there exists a 
pseudo-isomorphism $\alpha\colon X\dashrightarrow X_+$ such that $\Ind(\alpha) = C$ and the total image of $C$ under $\alpha$ is a smooth
rational curve $C_+$. Moreover, if we denote by $p\colon W\to X$ the blow-up of $X$ in $C$, then the exceptional divisor of $p$ is isomorphic to $\p^1\times\p^1$
and the blow-down of another ruling induces a regular birational morphism $p_+\colon W\to X_+$.

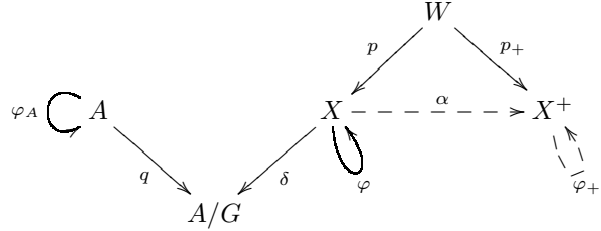
\begin{equation*}
 \xymatrix{
  &&& W \ar[ld]_{p} \ar[rd]^{p_+}&\\
 A\ar@(ul,dl)_{\varphi_A} \ar[rd]_{q} && X \ar@(d,dr)_{\varphi} \ar[ld]^{\delta} \ar@{-->}[rr]^{\alpha} && X^+\ar@{-->}@(d,dr)_{\varphi_+}\\
 & A/G &&&
 }
\end{equation*}

By the choice of the matrix $M$ the orbit of the curve $C$ is infinite and the intersection number $\theta_1(\varphi)\cdot [C]$ is strictly positive. 
Then we can apply the following lemma.

\begin{lemma}\label{lemma_intersections_after_flop}
 Let $\alpha\colon X\dashrightarrow X^+$ be a flop in a curve $C$. Then for any $D\in N^1(X)$ such that $ D\cdot [C]>0$ we have~\mbox{$(\alpha^*D)\cdot [C^+]<0$}.
\end{lemma}

In the case of the standard Atiyah flop this lemma is an easy computation. Since $\theta_1(\varphi_+) = (\alpha^{-1})^*\theta_1(\varphi)$ 
by Lemma \ref{lemma_pb_of_theta} we get the following inequality:
\begin{equation*}
 \theta_1(\varphi_+)\cdot [C_+]<0.
\end{equation*}
By the construction the curve $C_+$ lies in the indeterminacy locus $\Ind(\varphi^n)$ for all non-zero integers $n$. Thus, the pseudo-automorphism~$\varphi_+$ 
satisfies all but last properties of Condition \ref{condition}.

 \appendix
 \section{Computations for Lemma \ref{lemma_preimage_of_point}}
 In the proof of Lemma \ref{lemma_preimage_of_point} we need to compute the preimage of the point in the projective space $\p^3_{R[X, Y]}$ 
 described in Notation \ref{notation_point}. Here is the Sage code we used to perform the computations. 
 We define birational involutions $\sigma_j =$\lstinline{Ij} for $j = 1,2$ and $3$ and their composition $F_R^{-1} =$\lstinline{F_inverse}.
 Variables \lstinline{M0, N0, M1, N1, g2, gg2, g3} and \lstinline{gg3} play roles of $M_0, N_0, M_1, N_1, \widetilde{g}_2, g_2, 
 \widetilde{g}_3$ and $g_3$ respectively. Variables \lstinline{p2} and \lstinline{p3} represents elements $\widetilde{g}_2+2g_2$
 and $\widetilde{g}_3 +2g_3$ of the ideal $\I$.
 Then we apply \lstinline{F_inverse} to the point $p =$\lstinline{(M0+2*N0,M1+2*N1,g2+ 2*gg2,g3+2*gg3)} and get set of for polynomials \lstinline{Q1}.
 Also we apply \lstinline{F_inverse} to the point $p =$\lstinline{(M0+2*N0,M1+2*N1,p2,p3)} and get \lstinline{Q2}. 
\begin{lstlisting}[breaklines]
K.<a_2001, a_1200, a_1110, a_1101, a_0210, M0, M1, N0, N1, g2, gg2, g3, gg3, x0, x1, x2, x3, p2,p3>=ZZ[]

f = a_2001*x0^2*x3 + a_1200*x0*x1^2 + a_1110*x0*x1*x2 + a_1101*x0*x1*x3 + a_0210*x1^2*x2  

dfdx1 =f.derivative(x1)
dfdx2 =f.derivative(x2)
dfdx3 =f.derivative(x3)
    
def I1(x0,x1,x2,x3):
    return x0*dfdx1(x0=x0,x1=x1,x2=x2,x3=x3), x1*dfdx1(x0=x0,x1=x1,x2=x2,x3=x3)- 2*f(x0=x0,x1=x1,x2=x2,x3=x3), x2*dfdx1(x0=x0,x1=x1,x2=x2,x3=x3), x3*dfdx1(x0=x0,x1=x1,x2=x2,x3=x3)
def I2(x0,x1,x2,x3):
    return x0*dfdx2(x0=x0,x1=x1,x2=x2,x3=x3), x1*dfdx2(x0=x0,x1=x1,x2=x2,x3=x3), x2*dfdx2(x0=x0,x1=x1,x2=x2,x3=x3) - 2*f(x0=x0,x1=x1,x2=x2,x3=x3), x3*dfdx2(x0=x0,x1=x1,x2=x2,x3=x3)
def I3(x0,x1,x2,x3):
    return x0*dfdx3(x0=x0,x1=x1,x2=x2,x3=x3), x1*dfdx3(x0=x0,x1=x1,x2=x2,x3=x3), x2*dfdx3(x0=x0,x1=x1,x2=x2,x3=x3), x3*dfdx3(x0=x0,x1=x1,x2=x2,x3=x3) - 2*f(x0=x0,x1=x1,x2=x2,x3=x3)

def F_inverse(x0,x1,x2,x3):
    return(I1(*I2(*I3(x0,x1,x2,x3))))

Q1 = F_inverse(M0+2*N0,M1+2*N1,g2+ 2*gg2,g3+2*gg3)
Q2 = F_inverse(M0+2*N0,M1+2*N1,p2,p3)
\end{lstlisting}
Set of polynomials \lstinline{Q} is exactly the set of polynomials $(q_0, q_1, q_2, q_3)$ from the proof of Lemma \ref{lemma_preimage_of_point}.
To get the result we need several additional functions. The function \lstinline{mod_2I} deletes monomials in a polynomial which lie in 
the ideal $2\I$. 
\begin{lstlisting}
def mod_2I(P):
    Q = 0
    for m in P.monomials():
        c= P.monomial_coefficient(m)
        d2 = m.degree(p2)
        d3 = m.degree(p3)
        if d2+d3==2:
            # here we factor mod (pi*pj) \subset 2I
            Q = Q 
        if d2+d3 ==1:
            # here we factor mod (2pi) \subset 2I
            c2 = c % 2
            Q = Q + c2*m
        if d2+d3 ==0:
            # here we factor mod (8) \subset 2I
            c8 = c % 8
            Q = Q+c8*m
    return Q
\end{lstlisting}
The function \lstinline{mod_4} deletes all monomials in a polynomial which lie in 
the ideal $4R[X, Y]$.
\begin{lstlisting}
 def mod_4(P):      
    Q = 0
    for m in P.monomials():
        c= P.monomial_coefficient(m)
        d2 = m.degree(g2)
        d3 = m.degree(g3)
        dd2 = m.degree(gg2)
        dd3 = m.degree(gg3)
        if d2+d3+dd2+dd3 ==1:
            # here we factor mod (2gi) \subset (4)
            c2 = c % 2
            Q = Q = Q+c2*m
        if d2+d3+dd2+dd3 ==0:
            # here we factor mod (4) \subset (4)
            c4 = c % 4
            Q = Q+c4*m
    return Q 
\end{lstlisting}
The function \lstinline{mod_8} deletes all monomials in a polynomial which lie in 
the ideal $8R[X, Y]$.
\begin{lstlisting}
 def mod_8(P):
    Q = 0
    for m in P.monomials():
        c= P.monomial_coefficient(m)
        d2 = m.degree(g2)
        d3 = m.degree(g3)
        dd2 = m.degree(gg2)
        dd3 = m.degree(gg3)
        if d2+d3+dd2+dd3==2:
            # here we factor mod (2*gi*gj) \subset (8)
            c2 = c % 2
            Q = Q +c2*m
        if d2+d3+dd2+dd3 ==1:
            # here we factor mod (4*gi) \subset (8)
            c4 = c % 4
            Q = Q + c4*m
        if d2+d3+dd2+dd3 ==0:
            # here we factor mod (8) \subset (8)
            c8 = c % 8
            Q = Q+c8*m
    return Q    
\end{lstlisting}
The function \lstinline{leading_term_M1} returns the leading term in $M_1$ of the polynomial.
\begin{lstlisting}
 def leading_term_M1(P):
    Q=0
    d = P.degree(M1)
    for m in P.monomials():
        if d == m.degree(M1):
            c = P.monomial_coefficient(m)
            Q = Q + c*m
    return Q
\end{lstlisting}
The last thing we need is a function \lstinline{factorization} which factors polynomials into a product of irreducible polynomials. 
Unfortunately the computation above results in so-called symbolic functions, so to factor them we transform the symbolic function 
into a polynomial and then use the factorization in polynomials.
\begin{lstlisting}[breaklines]
b_2001, b_1200, b_1110, b_1101, b_0210, Z0, Z1, W0, W1,  f0, f1, f2, f3, ff2, ff3, q2,q3 = PolynomialRing(RationalField(), 17, ['a_2001', 'a_1200', 'a_1110', 'a_1101',  'a_0210', 'M0', 'M1', 'N0', 'N1', 'g0', 'g1', 'g2', 'g3', 'gg2', 'gg3', 'p2', 'p3']).gens()
 
def factorization(P):
    Q=0
    for m in P.monomials():
        c= P.monomial_coefficient(m)
        d2001 = m.degree(a_2001)
        d1200 = m.degree(a_1200)
        d1110 = m.degree(a_1110)
        d1101 = m.degree(a_1101)
        d0210 = m.degree(a_0210)
        dx0 = m.degree(M0)
        dx1 = m.degree(M1)
        dy0 = m.degree(N0)
        dy1 = m.degree(N1)
        d2 = m.degree(g2)
        d3 = m.degree(g3)
        dd2 = m.degree(gg2)
        dd3 = m.degree(gg3)
        dp2 = m.degree(p2)
        dp3 = m.degree(p3)
        Q = Q+ c * b_2001^d2001 * b_1200^d1200 * b_1110^d1110 * b_1101^d1101 * b_0210^d0210 * f2^d2 * f3^d3 * ff2^dd2 *ff3^dd3 * Z0^dx0 * Z1^dx1 * W0^dy0 *W1^dy1 * q2^dp2 * q3^dp3
    return Q.factor()
\end{lstlisting}
Now we are ready to get results:
\begin{lstlisting}
M0new = mod_4(Q1[0])
M1new = mod_4(Q1[1])
Q2_new = mod_2I(Q2[2])
Q3_new = mod_2I(Q2[3])
G2new = mod_8(Q1[2])
G3new = mod_8(Q1[3])


print('Leading term of 2*M_0_new = ', factorization(leading_term_M1(M0new)))
print('Leading term of 2*M_1_new = ', factorization(leading_term_M1(M1new)))
print('Q2_new modulo 2*I = ', factorization(Q2_new))
print('Q3_new modulo 2*I = ', factorization(Q3_new))
print('Leading term of 2*G2new = ', factorization(leading_term_M1(G2new)))
print('Leading term of 2*G3new = ', factorization(leading_term_M1(G3new)))
\end{lstlisting}
We get six lines of results. The first line produces the leading term in $M_1$ for $2M_0'$, 
we use this formula in \eqref{eq_M0_new}. The second line produces the leading term in $M_1$ for $2M_1'$, 
we use this formula in \eqref{eq_M1_new}. The third line produces the expression for $q'_2$, we use it in \eqref{eq_q2_in_I}.
The fourth line produces the expression for $q'_3$, we use it in \eqref{eq_q3_in_I}.
The fifth and sixth lines produce the leading terms in $M_1$ of $\widetilde{g}'_2$ and $\widetilde{g}'_3$ respectively,
we use them in \eqref{eq_g2_new} and \eqref{eq_g3_new}.

\bibliographystyle{alpha}
\bibliography{dynamics}

\begin{thebibliography}{BDPP13}

\bibitem[BC16]{Blanc-Cantat}
J.~Blanc and S.~Cantat.
\newblock Dynamical degrees of birational transformations of projective
  surfaces.
\newblock {\em J. Amer. Math. Soc.}, 29(2):415--471, 2016.

\bibitem[BCK14]{Bedford_Cantat_Kim}
E.~Bedford, S.~Cantat, and K.~Kim.
\newblock Pseudo-automorphisms with no invariant foliation.
\newblock {\em J. Mod. Dyn.}, 8(2):221--250, 2014.

\bibitem[BDPP13]{Movable_cone}
S.~Boucksom, J.-P. Demailly, M.~P\u{a}un, and T.~Peternell.
\newblock The pseudo-effective cone of a compact {K}\"{a}hler manifold and
  varieties of negative {K}odaira dimension.
\newblock {\em J. Algebraic Geom.}, 22(2):201--248, 2013.

\bibitem[BK09]{Bedford-Kim_ex}
E.~Bedford and K.~Kim.
\newblock Dynamics of rational surface automorphisms: linear fractional
  recurrences.
\newblock {\em J. Geom. Anal.}, 19(3):553--583, 2009.

\bibitem[BK14]{Bedford-Kim_Pseudo}
E.~Bedford and K.~Kim.
\newblock Dynamics of (pseudo) automorphisms of 3-space: periodicity versus
  positive entropy.
\newblock {\em Publ. Mat.}, 58(1):65--119, 2014.

\bibitem[Bla08]{Blanc_ex}
J.~Blanc.
\newblock On the inertia group of elliptic curves in the {C}remona group of the
  plane.
\newblock {\em Michigan Math. J.}, 56(2):315--330, 2008.

\bibitem[Bla13]{Blanc_Pseudo}
J.~Blanc.
\newblock Dynamical degrees of (pseudo)-automorphisms fixing cubic
  hypersurfaces.
\newblock {\em Indiana Univ. Math. J.}, 62(4):1143--1164, 2013.

\bibitem[CDX21]{Cantat_Deserti_Xie}
S.~Cantat, J.~D\'{e}serti, and J.~Xie.
\newblock Three chapters on {C}remona groups.
\newblock {\em Indiana Univ. Math. J.}, 70(5):2011--2064, 2021.

\bibitem[DF01]{Diller_Favre}
J.~Diller and C.~Favre.
\newblock Dynamics of bimeromorphic maps of surfaces.
\newblock {\em Amer. J. Math.}, 123(6):1135--1169, 2001.

\bibitem[DF20]{Dang_Favre_1}
N.-B. Dang and C.~Favre.
\newblock Spectral interpretations of dynamical degrees and applications.
\newblock {\em arXiv:2006.10262}, 2020.

\bibitem[DN11]{Dinh_Nguyen}
T.-C. Dinh and V.-A. Nguy\^{e}n.
\newblock Comparison of dynamical degrees for semi-conjugate meromorphic maps.
\newblock {\em Comment. Math. Helv.}, 86(4):817--840, 2011.

\bibitem[DO88]{Dolgachev_Ortland}
I.~Dolgachev and D.~Ortland.
\newblock Point sets in projective spaces and theta functions.
\newblock {\em Ast\'{e}risque}, (165):210 pp. (1989), 1988.

\bibitem[DS05]{Dinh_Sibony}
T.-C. Dinh and N.~Sibony.
\newblock Une borne sup\'{e}rieure pour l'entropie topologique d'une
  application rationnelle.
\newblock {\em Ann. of Math. (2)}, 161(3):1637--1644, 2005.

\bibitem[Ful98]{Fulton_IT}
W.~Fulton.
\newblock {\em Intersection theory}, volume~2 of {\em Ergebnisse der Mathematik
  und ihrer Grenzgebiete. 3. Folge. A Series of Modern Surveys in Mathematics
  [Results in Mathematics and Related Areas. 3rd Series. A Series of Modern
  Surveys in Mathematics]}.
\newblock Springer-Verlag, Berlin, second edition, 1998.

\bibitem[Har77]{Hartshorne}
R.~Hartshorne.
\newblock {\em Algebraic geometry}.
\newblock Springer-Verlag, New York-Heidelberg, 1977.
\newblock Graduate Texts in Mathematics, No. 52.

\bibitem[HM10]{Flips_and_flops}
C.~D. Hacon and J.~McKernan.
\newblock Flips and flops.
\newblock In {\em Proceedings of the {I}nternational {C}ongress of
  {M}athematicians. {V}olume {II}}, pages 513--539. Hindustan Book Agency, New
  Delhi, 2010.

\bibitem[KM98]{Kollar_Mori}
J.~Koll\'{a}r and S.~Mori.
\newblock {\em Birational geometry of algebraic varieties}, volume 134 of {\em
  Cambridge Tracts in Mathematics}.
\newblock Cambridge University Press, Cambridge, 1998.
\newblock With the collaboration of C. H. Clemens and A. Corti, Translated from
  the 1998 Japanese original.

\bibitem[Kol07]{Kollar_resolution}
J.~Koll\'{a}r.
\newblock {\em Lectures on resolution of singularities}, volume 166 of {\em
  Annals of Mathematics Studies}.
\newblock Princeton University Press, Princeton, NJ, 2007.

\bibitem[LB19]{Lo_Bianco_kahler_threefolds}
F.~Lo~Bianco.
\newblock On the cohomological action of automorphisms of compact {K}\"{a}hler
  threefolds.
\newblock {\em Bull. Soc. Math. France}, 147(3):469--514, 2019.

\bibitem[Les18]{Lesieutre_constraints}
J.~Lesieutre.
\newblock Some constraints of positive entropy automorphisms of smooth
  threefolds.
\newblock {\em Ann. Sci. \'{E}c. Norm. Sup\'{e}r. (4)}, 51(6):1507--1547, 2018.

\bibitem[OT15]{Oguiso-Truong}
K.~Oguiso and T.~T. Truong.
\newblock Explicit examples of rational and {C}alabi-{Y}au threefolds with
  primitive automorphisms of positive entropy.
\newblock {\em J. Math. Sci. Univ. Tokyo}, 22(1):361--385, 2015.

\bibitem[PS14]{PS-regularization}
Yu. Prokhorov and C.~Shramov.
\newblock Jordan property for groups of birational selfmaps.
\newblock {\em Compos. Math.}, 150(12):2054--2072, 2014.

\bibitem[Tru14]{Truong}
T.~T. Truong.
\newblock The simplicity of the first spectral radius of a meromorphic map.
\newblock {\em Michigan Math. J.}, 63(3):623--633, 2014.

\bibitem[Tru20]{Truong_rel_deg}
T.~T. Truong.
\newblock Relative dynamical degrees of correspondences over a field of
  arbitrary characteristic.
\newblock {\em J. Reine Angew. Math.}, 758:139--182, 2020.

\bibitem[Wei55]{Weil_reg}
A.~Weil.
\newblock On algebraic groups of transformations.
\newblock {\em Amer. J. Math.}, 77:355--391, 1955.

\end{thebibliography}

\end{document}